\documentclass[12pt,a4paper]{article}
\usepackage[english]{babel}
\usepackage{amsmath,amssymb,graphicx}
\usepackage{comment}
\usepackage{color}
\usepackage{fullpage} %style
\usepackage{hyperref}
\usepackage{mathtools}
\numberwithin{equation}{section}
\usepackage{shuffle}
\usepackage[utf8]{inputenc}
\usepackage{amsmath}
\usepackage{amsfonts}
\usepackage{amssymb}
\usepackage{amsthm}
\usepackage{color}
\usepackage{enumerate}

\def\bR{\mathbb{R}}

\def\bE{\mathbb{E}}

\def\diff{\mathop{}\!\mathrm{d}}
\def\qshuffle{\,\widehat{\shuffle}\,}

\theoremstyle{definition}
\newtheorem{definition}{Definition}[section]

\newtheorem{example}[definition]{Example}
\theoremstyle{plain}
\newtheorem{theorem}[definition]{Theorem}
\newtheorem{proposition}[definition]{Proposition}

\newtheorem{corollary}[definition]{Corollary}
\theoremstyle{remark}
\newtheorem{remark}[definition]{Remark}

\title{Generalized Euler-Maclaurin formula and Signatures} 

\date{}

\author{
 Carlo Bellingeri  \thanks{ \texttt{carlo.bellingeri@univ-lorraine.fr}}\\
  IECL, Unversité de Lorraine
  \and
 Peter K. Friz \thanks{\texttt{friz@math.tu-berlin.de}}\\
 TU Berlin \& WIAS
 \and
Sylvie  Paycha 
\thanks{\texttt{paycha@math.uni-potsdam.de}}\\
Universit\"at Potsdam
}

\begin{document}
	
	\maketitle

	\begin{abstract}
	   The  Euler-Maclaurin formula  which relates a discrete sum with an integral, is generalised to the setting of Riemann-Stieltjes sums and integrals on stochastic processes whose paths are a.s. rectifiable,  namely, continuous and with bounded variation. For this purpose, new variants of the signature are introduced, such as the flip and the sawtooth signature.  The counterparts of the Bernoulli numbers that arise in the classical  Euler-Maclaurin formula  are shown to be   the integration constants in the  repeated integration by parts which ``recursively minimise the error'' at every truncation level.
	\end{abstract}
 \medskip
 

\medskip

	\tableofcontents

	\section{Introduction}
	%\subsection*{The  classical Euler-Maclaurin formula}
The Euler-Maclaurin (EML) formula,  which compares the  sum  of the values of a function  $f\colon [0,N]\to \mathbb{R}$, with its integral, arises in various contexts and many disguises. Written in its original form, it states that for any function $f\in C^m([0,N])$   and any integer $m\geq 0$   the following identity holds
\begin{equation}\label{eq:clas_EML}
\sum_{k\in [[0, N]]^{\pm}}f(k)=\int_0^Nf(s)\diff s+ \sum_{l=1}^m \frac{B_l^{\pm}}{l!}[f^{(l-1)}(x)]_{x=0}^N+ R_m(f)\,.
\end{equation}
 In what follows $[[0, N]]^{-}= [[0,N-1]]$ denotes the index set for  the backward sum,  $[[0, N]]^{+}= [[1,N]]$ the index set for the forward sum and the sequence $\{B_l^{\pm}\}_{l\geq 0}$  stands for  the \emph{Bernoulli numbers}:
\[B^{\pm}_0=1\,,\quad B^{\pm}_1=\pm \frac{1}{2}\,,\quad  B^{\pm}_2= \frac{1}{6}\,, \quad B^{\pm}_4= - \frac{1}{30}\,,\quad B^{\pm}_6= \frac{1}{42}\,, \quad  B^{\pm}_8= \ldots,\]
\[B^{\pm}_3=B^{\pm}_5= B^{\pm}_7= \ldots =0,\]
see \cite{Saalschtz1893} for a full account on  different ways of deriving them. Furthermore, the remainder  has  the explicit integral form (whether for negative and positive Bernoulli numbers):
\begin{equation}\label{eq:clas_rem}
R_m(f)=\frac{(-1)^{m+1}}{m!}\int_0^Nf^{(m)}(s)P_m(s-[s])\diff s.\end{equation}
Here $\{P_m(s)\}_{m \geq 0}$ stands for  the family of Bernoulli polynomials
\[P_m(s)= \frac{\mathrm{d} ^m}{\mathrm{d} x^m}\frac{xe^{sx}}{e^x-1}|_{x=0}\]
and $[s]$ the integer part of $s$.

The EML formula is remarkable for many reasons and has a long history (see e.g. \cite{KacCheung} and references therein). It   was discovered independently by L. Euler and C. Maclaurin in
the first half of the XVIII-th century. However,  the essential remainder term $R_m$  was only later introduced by S.D. Poisson. in 1827, see \cite{Ostrowski1969}. 
 
 A  first consequence of the EML formula  is an exact formula for polynomial functions: indeed  for $f(x)=x^p$ with  $p\geq 0$ integer  Formula \eqref{eq:clas_EML} boils down to the Faulhaber formula:
\begin{equation}\label{Faul1}
\sum_{k\in [[0, N]]^{\pm}}k^p=\frac{1}{p+1}\sum_{j=0}^p\binom{p+1}{j}N^{p+1-j}B^{\pm}_j\,.
\end{equation}
This simpler identity can be also proven using different methods and it provides an equivalent  definition of the sequence $B^{\pm}$. Although the associated series appearing in \eqref{eq:clas_EML} is typically divergent,  the remainder is small in an appropriate sense which makes it useful in numerical mathematics and computer science, see \cite{Knuth94}. In addition to that, the EML formula  has several applications in number theory obtained for various choices of $f$. For $f(x)=x^{-1}$,  one obtains the well-known Euler-Mascheroni constant as the limit of the difference between the  harmonic series and the logarithm. For $f(x)=\log x$, one derives Stirling's approximation  formula of the factorial,   The EML formula applied to the function $f(x)=x^{-s}$ for $s\in \mathbb{C}$ yields a meromorphic extension of  the zeta function $s\mapsto \sum_{n=1}^\infty n^{-s}$ which is holomorphic on the half plane $\Re(s)>1$. Its generalisation to discrete sums on cones, which is relevant in   toric geometry in that it relates to the Todd genus, can be interpreted as an  algebraic Birkhoff factorisation which  plays a major role in  the context of renormalisation  \cite{Guo2015ACA}, \cite{guo2017}.

%\subsection*{{\color{purple}A wide range of proofs}}

Classical proofs  of the EML formula typically use repeated integration by part.  Some  express the integral in terms of a discrete sum e.g. \cite{lampret01}, others are based on a careful induction procedure, see e.g. \cite[Chapter 9]{Knuth94}.  

A proof  of the EML formula by Boas   \cite{Boas77} which  serves as  a useful starting point for what follows, expresses the discrete sum as an integral. In   the backward sum, we first extend   $f$ to $\mathbb R$ with a constant value and then write the  sum as an integral:
\[\sum_{k=0}^{N-1}f(k)=\int_{-a}^{N-a} f(t) \diff [t+1]\,,\] 
where as before, $[t]$ stands for the integer part of $t$ and $0<a<1$.  Setting $Q_1(t)=t-[t+1]+ c_1$ with $c_1\in \mathbb R$,  the difference between sum and integral can be written as
\[\sum_{k=0}^{N-1}f(k)-\int_{-a}^{N-a} f(t) \diff t=-\int_{-a}^{N-a} f(t) \diff (Q_1(t))\,.\]
Assuming $f$ is smooth,  a first integration by parts yields
\[\sum_{k=0}^{N-1}f(k)-\int_{-a}^{N-a} f(t) \diff t=(-a+c_1)\left( f(-a)-f(N-a)\right)+ \int_{-a}^{N-a} Q_1(t)  f'(t) \diff t\,.\]
Repeated   integration by parts  then yields
\begin{align*}
        &\sum_{k=0}^{N-1}f(k)-\int_{-a}^{N-a} f(t) \diff t=\\&(-c_1+a)[f(s)]_{s=-a}^{N-a} +\sum_{l=2}^{m} (-1)^l[f^{(l-1)}(s)Q_l(s)]^{N-a}_{s=-a}+(-1)^{m+1}\int_{-a}^{N-a} Q_{m}(t) f^{(m)}(t) \diff t\,,
\end{align*}
where $\{Q_l\}_{l\geq 2}$ is    any sequence of functions   such that $Q_l'= Q_{l-1}$. Formula \eqref{eq:clas_EML} then follows from picking the specific choice of periodic functions $Q_l(t)=\frac{1}{l!}P_l(t-[t])$ and letting $a$ tend to $0$. This choice is the unique one allowing  a sequence of periodic functions $\{Q_l\}_{l\geq 2}$ and turning  the terms $[f^{(l-1)}(s)Q_l(s)]^{N-a}_{s=-a}$ into $C_l[f^{(l-1)}(s)]^{N}_{s=0}$  for some constant  $\{C_l\}_{l\geq 2}$ which can be explicitly computed.

%	\subsection*{Extension to Stieltjes sums and integrals}

The purpose of this work is to take a fresh view on EML, replacing integration in $t$ by Riemann-Stieltjes integration against some rectifiable path $X$, i.e. continuous and finite variation, possibly the sample path of a rectifiable  process (a process whose trajectories are a.s. rectifiable, see e.g. \cite[Chapter 1]{protter2005stochastic})
\[ \int_0^N f (t) \diff t \rightsquigarrow \int_0^N f (X_t) \diff X_t\,.\]

The discrete counterparts on the level of sums are the \emph{backward} and \emph{forward Riemann-Stieltjes type sums}:
	\[I^{-}_N(f, X)= \sum_{k = 0}^{N-1} f (X_k) \delta X_k  \,,\quad I^{+}_N(f, X)=\sum_{k = 0}^{N-1} f (X_{k+1}) \delta X_k\,. \]
 where  $\delta X_k = X_{k+1}- X_k$.

The above sums and integrals are multidimensional, in that $X_t $ lies in a Banach space $ V$ and   $f\colon V \to L(V,W)$ is a smooth map in the  Fréchet differentiable sense, where $L(V,W)$ denotes the set of continuous linear maps from  $V$ to a Banach space $W$. 
We  further assume that  $f\colon V \to L(V,W)$ belongs to  the space $C^n(V ,L(V,W))$  of continuous maps  which are $n$ times continuously differentiable in the Fréchet sense.  For any integer $0\leq l\leq n$  we denote by $D^lf$ 
 the $l-$th Fréchet derivative of $f\in C^n(V ,L(V,W))$  which by construction can be represented as a map
\[D^lf\colon V \to L (V^{\otimes (k+1)}, W)\,.\]

In contrast with the  classical case, where $\delta X_k$ is constant,  the essential ingredients to derive an EML formula seem to be missing in the Stieltjes setting. Also, $X$ has a priori no second derivative, therefore we are in no position to apply the usual EML formula to $t \mapsto f (X_t) \dot{X_t} $ (which further would entangle $f$ and $X$ in an	undesirable fashion). 

%To achieve  the desired proof we recall the  steps of 

However, much in the spirit of Boas' proof recalled above, in the  setting of rectifiable processes, both the backward and forward Riemann-Stieltjes sums can be represented by integrals.
Also,    a standard integration by parts identity is available on the class of functions under consideration. This will result from   introducing   a  path $Z$ with values in $T_1((V))$, the space of tensor series with constant term equal to $1$ (this will be made explicit later). Our main insight,  of interest even in a classical EML setting (where this idea is new to our knowledge) is that the precise EML formula, which amounts to iteratively choosing integration constants in the  repeated integration by parts, results from  ``minimising the error'', for every truncation level $m\geq 1$. This optimality condition, which can be formulated in n “Hilbert” generality, thereby including classes of square-integrable processes,  is the key to a general Riemann-Stieltjes EML formula.

	\subsection*{Main result and its by-products}
	
Our main result is a generalised Euler-Maclaurin formula for a rectifiable process, which  involves sawtooth signatures. 
\begin{theorem}\label{main_thm_1}
For any rectifiable  process $X\colon [0,N]\to V$ with values in a Banach space $V$,   there exist two elements $b^{\pm,*}\in T_1((V))$, $b^{\pm,*}=(1, b^{\pm,*}_1, \ldots) $ called \emph{optimal backward} and \emph{forward tensors}  and two rectifiable  processes $Z^{\pm,*}(X)\colon [0, N]\to T_1((V))$, $Z^{\pm, *}(X)=(1,Z^{\pm,1,*}(X), \ldots)$  called \emph{optimal backward} and \emph{forward sawtooth signatures}  given by  Definition \ref{defn_optimal} such that $Z^{\pm,1,*}_0(X)= -b^{\pm,*}_1 $, $Z^{\pm,l,*}_0(X)= b^{\pm,*}_l$ and the following \emph{generalised EML formula} holds:
\begin{equation}\label{eq:EML_main_formula}
I_N^{\pm} (f, X) = \int_0^Nf(X_t)\diff X_t+[f(X_t)b^{\pm,*}_1]_{t=0}^N +\sum_{l = 2}^{m}(- 1)^{l}[D^{l-1}f(X_t) Z^{\pm, l,*}_t(X)]_{t=0}^N+ R^{m,\pm,*}(f, X)
\end{equation}
for any $f\in C^{m}(V, L(V,W))$ with $m\geq 1$, where the remainder is given by the Stieltjes integral 
\begin{equation}\label{eq:rem_EML}
   R^{m,\pm,*}(f, X)=(- 1)^{m+1}\int_{0}^{N}  D^mf(X_t)  \diff  Z^{\pm,m+1, *}(X)\,.
\end{equation}
 When $V$ is a Hilbert space, $b^{\pm,*}$ and $Z^{\pm,*}(X)$ are up to a sign factor the unique solution of the recursive minimisation problem  for the remainder in Theorem \ref{prop_optimisation_P}.
\end{theorem}
\begin{proof}
    The theorem follows from    combining Theorem \ref{thm:prelEML} with $Y=f(X)$    with Definition \ref{defn_optimal}.
\end{proof}

\begin{remark}
According to the main identity \eqref{eq:EML_main_formula} the constants $b^{\pm,*}$    each have their interpretation in terms of backward and forward sums,  which is why we distinguish between positive and negative Bernoulli numbers and  treat them separately   throughout the paper, in spite of  a recent trend to suppress   negative Bernoulli numbers, see
 \cite{Bernou_manif}.
\end{remark}
\begin{remark}

The integrality of the end points $0$, $N$ of the domain of integration plays no special role here and  one could equally well consider, for any interval with  end points $a < b$ and partition $\pi$ thereof, integrals of the type
\[ \int_a^b f(X_t) d X_t  \quad \text{vs} \quad \sum_{[s, t] \in \pi} f(X_s)  (X_t - X_s) \, . \]
This can be trivially achieved by a time change. Moreover, our choice of equidistant  partition $\pi$ (given by integers) suggests that one might want to consider $f(X)$ and $X$  in the natural times scales and hence to
parametrise $X$ by the arclength
\[ t \mapsto \| X \|_{1 - \text{var} ; [0, t]}= \omega(t) . \]
The resulting time-changed path $\tilde{X}=X\circ \omega^{-1}$ is   Lipschitz continuous, hence also absolutely continuous, and runs, by construction, at constant speed $v$. 
Since length is unaffected by reparametrisation, we have \[ \left\| v \right\| (b - a) = \| X \|_{1 - \text{var} ; [a, b]},\]
which determines $v$,   an observation which  might simplify the computations in applications. 
\end{remark}

\begin{remark}
Summing the forward and backward formula of \eqref{eq:EML_main_formula} and dividing by $2$, it  further yields an expression for the trapezoidal Riemann sum
\[I_N^{1/2} (f, X)=  \frac{1}{2}\sum_{k = 0}^{N-1} (f (X_k)+ f(X_{k+1})) \delta X_k\,. \]
Unfortunately, in the full generality of Theorem \ref{main_thm_1}, considering such sums does not necessarily imply a simplification of \eqref{eq:EML_main_formula}. For the standard  trapezoid Riemann sums, one has a classical simplification due to the fact that $B_1^{+}=-B_1^{-}$.
\end{remark}

%{\color{red} Perhaps in relation to this point, I think that the role of probability theory in the paper should be made clearer. Is the main theorem not a purely deterministic statement? If it is not, why?(Referee question)}
\begin{remark}
   Although formulated using stochastic processes, our result is essentially deterministic because it uses only trajectorial properties  of the underlying process $X$  from which we compute an average in Theorem \ref{prop_optimisation_P}. If  the underlying process $X$ is not rectifiable but it has  continuous trajectories with finite  $p$-variation with $1\leq p<2$, all the constructions related to the existence of integration by part formula and the existence of the sawtooth signature  literally extend to the context of Young integration \cite{Young1936}. However,  the optimisation procedure in Theorem \ref{prop_optimisation_P} cannot be applied and  the remainder  term  \eqref{eq:rem_EML} appearing in this context cannot be controlled using H\"older inequality but using instead  Young's maximal inequality.
\end{remark}

	%\subsection*{By-products of the main result}
As a by-product of our generalised EML formula, we   revisit the classical EML formula showing how the Bernoulli numbers arise from an optimisation procedure. To to do,    for any sequence of numbers $b=(b_0\,,b_1\,, \ldots )$ we introduce the sequence of sawtooth functions $s^{\pm}_b= (s^{\pm,0}_b\,,s^{\pm,1}_b\,, \ldots)$
 with $s^{\pm,0}_b$ constant and equal to $1$.  For any  integer $l \geq 1$ we set
\begin{align*}  
s^{+,l}_b(t)&=\sum_{j=0}^{l}\frac{t^{l-j}}{(l-j)!}b_j  -  \sum_{k=0}^{[t]-1}\frac{(t-k-1)^{l-1}}{(l-1)!}\,, \\
s^{-,l}_b(t)&= \sum_{j=0}^{l}\frac{t^{l-j}}{(l-j)!}b_j - \sum_{k=0}^{[t]}\frac{(t-k)^{l-1}}{(l-1)!}\,.
\end{align*}

\begin{corollary}\label{cor:EMLoptimal}
The sequence of Bernoulli numbers $B=\{B_l^{\pm}\}_{l\geq 0}$ can be  equivalently defined in the following recursive way: set $B_0^{\pm}=1$, then  choose $B_1^{\pm}$ by
minimising respectively the quadratic functions
\[v\in \mathbb{R} \to\int_0^N \vert t-[t] - v\vert^2\mathrm{d} t\,, \quad v\in \mathbb{R} \to\int_0^N \vert t-[t+1] - v\vert^2\mathrm{d} t\,.\] 
Then, for any integer $l\geq 2$ the term  $B_{l}^{\pm}$ is obtained  by minimising the quadratic function
\[v\in \mathbb{R} \to\int_0^N \vert s^{\pm,l}_{B^{\pm,< l}}(t)+v\vert^2\mathrm{d} t\,, \] 
where $B^{\pm,< l}=(B_0^{\pm}\,,-B_1^{\pm} \ldots\,, B_{l-1}^{\pm}, 0, 0, \ldots )$.
\end{corollary}
\begin{proof}
    The theorem follows from Theorem \ref{prop_optimisation_P}, Remark \ref{rk_minus_sign} and Theorem \ref{thm:classical_bernoulli}.
\end{proof}
On the way to Theorem \ref{main_thm_1} we also obtain an explicit formula (Theorem \ref{thm:correction_signature2})
relating the \emph{signature} of $X$ (see \cite{Hambly2010-zu}), $S(X)\colon \Delta^2_{0,N}\to T_1((V))$,
given on each component $S(X)=(1,S^1(X), \ldots)$ by 
\begin{equation}\label{def_sign}
S^n_{s,t}(X)= \int_{\Delta_{s,t}^n} \diff X_{t_1}\otimes \ldots \otimes \diff X_{t_n}\,,
\end{equation}
 where we use the convention $\Delta_{s,t}^n=   \{(t_1,\ldots t_n)\in [0,N]^n \colon s<t_1<\ldots <t_n<t \}$,   with the corresponding  \emph{sum signature} 
 $\Sigma(X)\colon [[0, N]]^2_<\to T_1((V))$, defined from $[[0, N]]^2_<= \{(m,k)\in \{0,\ldots, N\}^2\colon m<k\}$ to each component $\Sigma(X)=(1,\Sigma^1(X), \ldots)$ by 
\begin{equation}\label{def_sum_sign}
\Sigma^n_{m,k}(X)= \sum_{m\leq i_1<\ldots <i_n<k}\delta X_{i_1}\otimes \ldots \otimes \delta X_{i_n}\,.
\end{equation}
 In this formula,  the sawtooth signature serves to interpolate  the   continuous and discrete signatures with the latter as  in equation \eqref{eq:discrete_signature}. In the finite dimensional setting, i.e. when $V=\mathbb{R}^d$ we show in  Theorem \ref{last_thm} that the identity \eqref{eq:discrete_signature}  provides an alternative way to compute the signature of the linear interpolation of the values of a time series  in term of the \emph{iterated sum signature}, see \cite{Tapia20, bellingeri2023discrete}. This equality is well-known in the literature, see \cite[Theorem 5.3]{Tapia20} and \cite[Corollary 5.3]{Amendola2019}.

\subsection*{Structure of the paper}

 The paper is organised as follows:  We first introduce two new families of non-linear functionals associated with $X$, the \emph{flip signature} Definition \ref{defn:flipsign} and the \emph{sawtooth signature} (equation \eqref{def_saw_signature1} which relates to the flip signature by Theorem \ref{EML_signature}. We then derive a purely deterministic identity \eqref{eq:pre_EML_main_formula} on rectifiable paths $X:[0, N]\to V$ using the sawtooth signature.  More precisely, for a  $m$-controlled integrand $Y$ of  $X$  (Definition \ref{defn:pcontrolledint}), we relate the Riemann-Stieltjes sums with its Stieltjes integral in Theorem \ref{thm:prelEML}. We then derive the main EML formula \eqref{eq:EML_main_formula}  from Theorem \ref{thm:prelEML}  by imposing an optimality criterion to choose the initial constant $b$ independently of the underlying integrand $Y$.   More precisely, we show that the operation of minimising the variance of each component in the remainder term  \eqref{eq:rem_pre_EML} leads to a natural family of constants that coincide with the classical Bernoulli number.  Finally, we conclude by deducing an interpolation formula between the sum signature \eqref{def_sum_sign} and the continuous signature \eqref{def_sign} by mixing the sawtooth signature with the discrete sum in  Theorem \ref{thm:correction_signature2}. This identity will provide us an alternative proof  to compute the signature of a piecewise linear path (see Theorem \ref{last_thm}).   
 
 We stress that this new approach to the Euler-Maclaurin formula could inspire several new connections to other branches of stochastic analysis, especially in the computation of cubature formulae on Wiener space \cite{Lyons2004,Bayer2012} where several examples of rectifiable processes $X$ are defined. More generally, we would like to extend the EML formula when the underlying path $X$ is a more general driver like a rough path or a semimartingale.

\bigskip

\noindent \textbf{Acknowledgements:}
CB gratefully acknowledges funding support from the European Research Council (ERC) through the ERC Starting Grant Low Regularity Dynamics via Decorated Trees (LoRDeT). He also acknowledges support of TU Berlin where he was employed when this project was started. PKF and SP acknowledge support from DFG CRC/TRR 388 Rough Analysis, Stochastic Dynamics and Related Fields", Project A05. The authors thank Sichuan University in Chengdu where discussions around the paper took place during the conference ``Algebraic, analytic, geometric structures emerging from quantum field theory'' in March 2024.
\bigskip

\noindent \textbf{Notations:} All along the paper we will work with a fixed real Banach space $V$. Tensor powers $(V^{\otimes n})_{n \geq 1} $ of $V$  are taken starting from an admissible norm over $V^{\otimes n}$, see  e.g. \cite[Definition 1.25]{lyons2007differential} which  satisfies the continuity property 
\[\Vert v\otimes w\Vert_{V\otimes V}\leq \Vert v\Vert_{V} \Vert  w\Vert_{V}\]
and the invariance identity	
\[ \Vert v_{\sigma(1)}\otimes \cdots \otimes v_{\sigma(n)}\Vert_{V{\otimes n}}=\Vert v_1\otimes \cdots \otimes v_n\Vert_{V{\otimes n}}\,\]
for any permutation $\sigma\in S_n$. A standard example of an  admissible norm is the projective tensor product, see \cite[Theorem 2.9]{Ryan02}.  We denote by $T((V))$ the space of formal tensor series over $V$ and $T^m(V)$ the truncated tensor algebra up to degree $m$ for any integer $m \geq 0 $. The operation $\otimes$  naturally  turns $T ((V))$   into an algebra. Projections on $V^{\otimes n}$ will be denoted by $\pi_n$ and we will use the notations $T_1 ((V))$  (resp. $ T_1^p (V)$) to denote the subspaces of elements $v\in T ((V)) $ (resp. $v\in T^p (V)$) such that $\pi_0(v)=1$. The unity of the tensor product will be denoted by $\mathbf{1}$. This algebra equipped with the operation $\otimes$  is a group with inverse operation denoted by $v^{-1}$.
\section{Sawtooth and flip signature of a path}
 We introduce two new families of non-linear functionals associated to a rectifiable path $X \colon [0, N] \rightarrow V$. Both of them will look similar to the signature $S(X)$ for which the usual differential characterisation
\begin{equation*}
	\begin{cases}
	    d S_{0,t} (X)=  S_{0,t} (X)\otimes\diff X_t\\
     S_{0,0}(X)=\mathbf{1}.
	\end{cases}
	\end{equation*}
holds, see \cite{Hambly2010-zu}.  However, as it will be clear in section \ref{sec:3},  to prove Theorem \ref{main_thm_1} we will need to consider tensor-valued paths that are characterised by a similar relationship where the term $dX_t$ is placed to the left instead of the right.

\begin{definition}\label{defn:flipsign}
For any  rectifiable path $X \colon [0, N] \rightarrow V$, we define  the \emph{flip-signature} $ S^{\flat}(X)\colon \Delta^2_{0,N}\to T_1((V))$  of $X$ by 
\begin{equation} \label{eq:flipsign}
\pi_n S^{\flat}_{s,t}(X)= \int_{\Delta_{s,t}^n} \diff X_{t_n}\otimes \ldots \otimes \diff X_{t_1}\,. \end{equation} 
\end{definition}
	
	\begin{proposition}
 The \emph{flip signature} is characterised as the unique solution of the tensor valued differential equation
	\begin{equation}\label{eq:flip_Ode}
	\begin{cases}
	    d S^{\flat}_{0,t} (X)= \diff X_t\otimes S^{\flat}_{0,t} (X)\\
     S^{\flat}_{0,0}(X)=\mathbf{1}.
	\end{cases}
	\end{equation}		
%A flipped Chen relation holds 		\[{S^{\flat}} (X)_{a,b}= {S^{\flat}} (X)_{c,b}\otimes {S^{\flat}} (X)_{a,c}\,.\] 		For any $a,  b$  one has the identity   
Moreover we have $S^{\flat}_{s,t} (X)= \Gamma {S}_{s,t} (X)^{-1} $ where $ \Gamma\colon T((V))\to T((V))$ is the linear involution map defined on any $v\in T((V))$ by  $
\pi_n(\Gamma v)=(-1)^n \pi_n(v)\,.$
\begin{comment}
    \[ \pi_n(\Gamma v)=(-1)^n \pi_n(({S} (\overleftarrow{X})_{a,b}))=(-1)^n \pi_n(({S} (X)_{a,b})^{-1})= (-1)^n\pi_n({S} (X)_{b,a}) \]where $\overleftarrow{X}$ is the time-reverse path $\overleftarrow{X}_t=X_{a+b-t}$.
\end{comment}
\end{proposition}
 \begin{remark}
Note that the equation \eqref{eq:flip_Ode}  is obtained from the usual equation  satisfied by the  signature $S(X)$ by flipping both sides, hence the name flip signature.
 \end{remark}
	\begin{proof} 
We first prove the algebraic identity among flipped and inverse signature. Performing a standard substitution in the integration domain we write 
\[ \pi_n S^{\flat}_{s,t}(X)=\int_{\Delta_{s,t}^{n,*}}\diff X_{t_1}\otimes \cdots \otimes \diff X_{t_n},\]
where $\Delta_{s,t}^{n,*}= \{(t_1,\ldots t_n)\in [0,N]^n\colon s<t_n<\ldots <t_1<t \}$ is the reverse simplex. Introducing  the time reverse path $  \overleftarrow{X}\colon [s,t]\to V $ given by  $\overleftarrow{X}_{u}=X_{s+t-u} $ we have 
\[\pi_n S^{\flat}_{s,t}(X)=(-1)^n\int_{\Delta_{s,t}^n}d\overleftarrow{X}_{t_1}\otimes \cdots  \otimes \diff \overleftarrow{X}_{t_n}= (-1)^nS_{s,t} (\overleftarrow{X})\,.\]
Since $S_{s,t} (\overleftarrow{X})= S_{s,t} (X)^{-1}$ as consequence of the Chen identity \cite [Theorem 2.9]{Chen54}, we obtain the desired identity. Passing to  equation \eqref{eq:flip_Ode}, one obtains it starting from Chen's relation $S_{s,t} (X)= S_{s,u} (X)\otimes S_{u,t} (X)$, then after inversion and multiplication by $\Gamma$ we derive $S^{\flat}_{s,t} (X)= S_{u,t}^{\flat} (X)\otimes  S^{\flat}_{s,u}(X)$ and the equation \eqref{eq:flip_Ode} arises when we take $s=0$, we divide by $t-u$ and we take the limit $u\to t$.
\end{proof}

	\begin{comment}
	Induction on  $l\geq 2$
	\[\pi_2 (Z^{+}_t)=\int_0^{m} \diff X_s\otimes \pi_1 (Z^{+}_s)+ =\sum_{l=0}^{[t]-1}\int_l^{l+1} \diff X_s\otimes \pi_1 (Z^{+}_s)+ \int_{[t]}^t \diff X_s\otimes \pi_1 (Z^{+}_s)\]
	\[= \sum_{l=0}^{[t]-1}\int_l^{l+1} \diff X_s\otimes\left(X_s-X_l\right)+ \int_{[t]}^t \diff X_s\otimes \left(X_s-X_{[t]}\right)\,.\]

	\[\pi_2 (Z^{+}_t) = \int_{0}^{t} \diff X_s \otimes \left( X_s-X_0\right)- \int_{0}^{t} dX_s \otimes \left( X_{[s]}-X_0\right)\]\[=\pi_2 (S^{\flat}_{0,t}(X))- \int_{0}^{t} dX_s \otimes \left( X_{[s]}-X_0\right) \]
	\[\pi_2 (S^{\flat}_{0,t}(X))- \sum_{l=0}^{[t]-1}\int_{l}^{l+1} dX_s \otimes \left( X_{l}-X_0\right)-\int_{[t]}^{t} dX_s \otimes \left( X_{[t]}-X_0\right) \]

	$l$ whatever
	\[\pi_3 (Z^{+}_t)=\int_0^{t} dX_s\otimes \pi_{2} (Z^{+}_s)=  \int_0^{t} dX_s  \otimes \left(\sum_{l=0}^{[s]-1}\pi_2({S^{\flat}} (X)_{l,l+1})+\pi_2({S^{\flat}} (X)_{[s],s})\right)=\]
	\[\sum_{m=0}^{[t]-1} \int_m^{m+1} dX_s\otimes \left(\sum_{l=0}^{m-1}\pi_2({S^{\flat}} (X)_{l,l+1})+\pi_2({S^{\flat}} (X)_{m,s})\right)+ \]
	\[\int_{[t]}^{t}dX_s\otimes \left(\sum_{l=0}^{[t]-1}\pi_2({S^{\flat}} (X)_{l,l+1})+\pi_2({S^{\flat}} (X)_{[t],s})\right)\]
	\[\sum_{l=0}^{[t]-1}\left(X_{[t]-1}- X_{l+1}\right)\otimes\pi_2({S^{\flat}} (X)_{l,l+1})+\sum_{l=0}^{[t]-1}\left(X_t- X_{[t]}\right)\otimes \pi_2({S^{\flat}} (X)_{l,l+1}) \]\[+ \sum_{m=0}^{[t]-1}\pi_3({S^{\flat}} (X)_{m,m+1})+\pi_3({S^{\flat}} (X)_{[t],t})\]
	
	\end{comment}
	
\begin{example}
 When $V =\mathbb{R}$,  the tensor product $\otimes$ coincides with the usual product of real numbers and  we  recover the flip signature of any path $X\colon [0,N]\to \mathbb{R}$  via the classical formula 
\begin{equation}\label{eq:classicalflipsign}\pi_l(S^{\flat}_{s,t}(X))=\frac{(X_t-X_s)^{l}}{l!}\,,
\end{equation}
which for $X (t) = t$,  reads 
\[\pi_l(S^{\flat}_{s,t}(X))=\frac{(t-s)^{l}}{l!}\,.\] 
More generally, when $V=\bR^d$, we can compute each component of $S^{\flat}(X)$ via the straightforward identity
\[\langle S^{\flat}_{s,t}(X), i_1 \ldots i_n\rangle =\langle S_{s,t}(X), i_n\ldots i_1\rangle, \]
where we use the identification  
\[e_{i_1}\otimes \ldots \otimes e_{i_n}= i_1\ldots i_n.\]
Here $e_i, i=1, \cdots, d$ is the canonical basis of $\bR^d$ and the bracket $\langle\,,\rangle$ is the extension to $T((\mathbb{R}^d))$ of the standard scalar product of $\mathbb{R}^d$.
\end{example}
We now introduce the  key ingredient to describe the generalised EML formula.
\begin{definition}
For any given rectifiable path   $X\colon [0,N]\to V$ and $b\in T_1((V)) $ we consider the paths $Z^{\pm}(X, b)\colon [0,N]\to T_1((V))$ given by the first order condition
\begin{equation}\label{def_saw_signature1} 
\pi_1(Z^{+}_t(X,b))=X_t-X_{[t]}+ \pi_1(b)\,,\quad \pi_1(Z^{-}_t(X,b))=X_t-X_{[t+1]}+\pi_1(b)
\end{equation}
and   whose  higher order components for any integer $l\geq 2$ satisfy the differential equations
\begin{equation}\label{def_saw_signature2}
\begin{cases}
  d\pi_{l}(Z^{\pm}_t(X,b))=dX_t\otimes\pi_{l-1}(Z^{\pm}_t(X,b))   \\
 \pi_l(Z^{\pm}_0(X,b))= \pi_k(b)
\end{cases}
\end{equation}
We call the paths $Z^{\pm}(X, b)$ the \emph{forward} and  \emph{backward sawtooth signature} with initial datum $b$. When $b=\mathbf{1}$, we  l use the shorthand notation $Z^{\pm}(X,\mathbf{1})=  Z^{\pm}(X)$ and simply call it the \emph{forward} and \emph{backward sawtooth signature}. If $b\in T^m_1(V)$ we  call the projection of $Z^{\pm}_t(X,b )$ on $T^m_1(V)$ the \emph{truncated forward and backward sawtooth signature with initial datum} $b$, denoting it in the same way.
\end{definition}

A  similar procedure applied to the   differential relation  \eqref{eq:flip_Ode},
 whose structure is similar to that of \eqref{def_saw_signature2}, yields an  identity between the sawtooth and the flip-signature.
	
\begin{theorem}\label{EML_signature}
For any bounded variation  path $X$  and  $b\in T_1((V))$ we have
\begin{align}
Z^{+}_t(X,b)
&= S^{\flat}_{0,t}(X)\otimes b-\sum_{k=0}^{[t]-1}{S}^{\flat}_{k+1,t} (X)\otimes \delta X_k\,,
\label{thm_saw_signature_1}\\ Z^{-}_t(X,b)&=S^{\flat}_{0,t}(X)\otimes b-\sum_{k=0}^{[t]}{S}^{\flat}_{k,t}(X)\otimes\delta X_k\,. \label{thm_saw_signature_2}.
\end{align}
\end{theorem}
%is explicitly given by $Z^{*,1}_t=X_t-X_{[t+1]}$ and the general formula
%\[Z^{*,l}_t=\]\\& \pi_l(\bigotimes_{m=0}^{[t]-1}{S^{\flat}} (X)_{m,m+1}\otimes {S^{\flat}} (X)_{[t],t})-\sum_{m=0}^{[t]}\pi_{l-1}({S^{\flat}} (X)_{m,t})\otimes (X_{m+1}- X_m)
%where $X^{*,l}_{s,t}$ is defined over $\Delta^{l,*}_{s,t}$. We call $Z^*$ the \r{EML signature}.
%\begin{equation}\label{eq_first_component}
%      \pi_1 (Z^{+}_t)= X_t-X_{[t]} \,,  \quad \pi_1(Z^{-}_t)=X_t-X_{[t+1]}\,; 
%\end{equation}= \pi_l ({S^{\flat}} (X)_{[t],t})+ \sum_{k=0}^{[t]-1}\sum_{m=0}^{l-2}\pi_m ({S^{\flat}} (X)_{k+1,t})\otimes \pi_{l-m} ({S^{\flat}} (X)_{k,k+1})\,,\\& =\pi_l({S^{\flat}} (X)_{[t]+1,t})+\sum_{k=0}^{[t]}\sum_{m=0}^{l-2}\pi_m ({S^{\flat}} (X)_{k,t})\otimes (-1)^{l-m+1}\pi_{l-m} ({S^{\flat}} (X)_{k,k+1})\,.
\begin{proof}
We first prove the identities \eqref{thm_saw_signature_1} and \eqref{thm_saw_signature_2}  for the case $b=\mathbf{1}$. Projecting onto each component of $T_1((V))$, this amounts to showing that for any $l\geq 1 $ one has 
\begin{align}
\pi_l(Z^{+}_t(X,b))
&= \pi_l({S_{0,t}^{\flat}} (X))-\sum_{k=0}^{[t]-1}\pi_{l-1}({S}^{\flat}_{k+1,t} (X))\otimes \delta X_k\,,\label{proof_saw_signature_1}\\ \pi_l(Z^{-}_t(X))&= \pi_l({S_{0,t}^{\flat}} (X))-\sum_{k=0}^{[t]}\pi_{l-1}(S^{\flat}_{k,t}(X))\otimes \delta X_k\,,\label{proof_saw_signature_2}
\end{align}
for any $ t\in[0, N]$.  We  show these identities by induction on $l\geq 1$, focusing only on $Z^{-}$ for the sake of conciseness. 
		
Since by assumption  $\pi_0({S_{0,t}^{\flat}}(X))=1$ and $\pi_1({S _{0,t}^{\flat}}(X))=X_t-X_0$,   the case $l=1$ reads:
\begin{align*}
&\pi_1(S^{\flat}_{0,t}(X))-\sum_{k=0}^{[t]}\pi_{0}({S_{k,t}^{\flat}} (X))\otimes\delta X_k\\&= X_t- X_0-\sum_{k=0}^{[t]}\delta X_k = X_t- X_0- (X_{[t]+1}-X_0 )= \pi_1(Z^-_t(X)) \,,
\end{align*}
where we have used (\ref{def_saw_signature2}), thereby showing the  initial  induction step. For a generic $l>1$,  by definition of $Z^{-}_t(X)$ and using  the recursive assumption one has  
\begin{align*}
\pi_l (Z^{-}_t(X))&= \int_0^t  \diff X_s\otimes\pi_{l-1}(Z^{-}_s(X))\\&=\int_0^t  \diff X_s\otimes \pi_{l-1}({S_{0,s}^{\flat}} (X))-\int_0^t \sum_{k=0}^{[s]} \diff X_s\otimes\pi_{l-2}({S_{k,s}^{\flat}} (X))\otimes\delta X_k.
\end{align*}
Using the characterisation \refeq{eq:flip_Ode} of the flip signature, and exchanging the sum with the integral yields
\begin{align*}
\pi_l (Z^{-}_t(X))=&\pi_{l}({S^{\flat}_{0,t}} (X))-\sum_{k=0}^{[t]}\int_{k}^t \diff X_s\otimes\pi_{l-2}({S_{k,s}^{\flat}}(X))\otimes\delta X_k\\&=\pi_{l}({S^{\flat}_{0,t}} (X))-\sum_{k=0}^{[t]}\pi_{l-1}({S_{k,t}^{\flat}} (X))\otimes\delta X_k, 
\end{align*}
thereby closing the induction. 
		
For general $b$, we prove the identities  by showing that the path 
\[Y_t= Z^{-}_t(X,b)-Z^{-}_t(X)+S^{\flat}_{0,t}(X)\]
solves the same differential equation \eqref{eq:flip_Ode} as $S^{\flat}_{0,t}(X)$ with initial datum $b$. Indeed by uniqueness of the Ode \eqref{eq:flip_Ode} we  obtain
\[Z^{-}_t(X,b)-Z^{-}_t(X)+S^{\flat}_{0,t}(X)= (S^{\flat}_{0,t}(X)\otimes b) \]
from which the result will then follow.   Since on any component of $T(V)$ of order $k>1$ one has 
\begin{align}\label{eq_check_proof}
d\pi_{k+1}(Y_t)=\diff X_t\otimes\pi_k(Y_t)   \,,
\end{align}
and $Y_0=b$, it is sufficient to check \eqref{eq_check_proof} in the case $k=\mathbf{1}$ to conclude. But in this case we have
\[\pi_1(Z^{-}_t(X,b)-Z^{-}_t(X)+S^{\flat}_{0,t}(X))=X_t-X_0+ \pi_1(b)=\pi_1(S^{\flat}_{0,t}(X)\otimes b )\,.\]
\end{proof}

	\begin{example} When $V =\mathbb{R}$,   the flip signature of any path $X\colon [0,N]\to \mathbb{R}$ is given by \eqref{eq:classicalflipsign}, so that  for any initial datum $b\in T_1((\bR))$, the formulae of Theorem  \ref{EML_signature} become at each order
		\begin{align*}
  \pi_l(Z^{+}_t(X, b))&= \sum_{j=0}^{l}\frac{(X_t-X_0)^{l-j}}{(l-j)!}\pi_j(b)-\sum_{k=0}^{[t]-1}\frac{(X_t-X_{k+1})^{l-1}\delta X_k }{(l-1)!}\,
,\\	\pi_l(Z^{-}_t(X,b))&= \sum_{j=0}^{l}\frac{(X_t-X_0)^{l-j}}{(l-j)!}\pi_j(b)-  \sum_{k=0}^{[t]}\frac{(X_t-X_k)^{l-1}\delta X_k }{(l-1)!}\,.
		\end{align*}
		When $X_t = \lambda t$ for some $\lambda\in \mathbb{R}$ and $V= \mathbb{R}$, then $\delta X_k =\lambda$  for all $k$ and   one  has the simpler  expressions

\begin{align}  \label{expl_saw_sign1}
\pi_l(Z^{+}_t(\lambda t,b))&=\sum_{j=0}^{l}\frac{\lambda^{l-j}t^{l-j}}{(l-j)!}\pi_j(b)  -  \lambda^{l}\sum_{k=0}^{[t]-1}\frac{(t-k-1)^{l-1}}{(l-1)!}\,, \\
\label{expl_saw_sign2}
\pi_l(Z^{-}_t(\lambda  t,b))&= \sum_{j=0}^{l}\frac{\lambda^{l-j}t^{l-j}}{(l-j)!}\pi_j(b) - \lambda^{l} \sum_{k=0}^{[t]}\frac{(t-k)^{l-1}}{(l-1)!}\,.
\end{align}
With the help of the Faulhalber formula \ref{Faul1}, and the condition $b\in T_1((V))$ we write 
\begin{align}  \label{Bernoulli_expl_saw_sign1}
\pi_l(Z^{+}_N(\lambda t,b))&
= \sum_{j=1}^{l}\frac{\lambda^{l-j}N^{l-j}}{(l-j)!}\pi_j(b)- \lambda ^l\sum_{j=1}^{l-1}\frac{N^{l-j}}{(l-j)!}\frac{B^{-}_j}{j!}\\ 
\label{Bernoulli_expl_saw_sign2}
\pi_l(Z^{-}_N(\lambda N,b))&= \sum_{j=1}^{l}\frac{\lambda^{l-j}N^{l-j}}{(l-j)!}\pi_j(b) -\lambda ^l\left(\sum_{j=1}^{l-1}\frac{N^{l-j}}{(l-j)!}\frac{B^{+}_j}{j!}+ \delta_{l1}\right)\,,
\end{align}
where $\delta_{l1}$ is the Kronecker delta.
\end{example}
\section{A preliminary deterministic EML formula}\label{sec:3}
We use the sawtooth signature to derive a first deterministic identity on rectifiable paths $X\colon [0,N]\to V$  where no optimisation procedure is involved. In order to develop a framework which comprises the two main theorems, we introduce a class of paths which generalise  those of the form $f(X_t)$. 
\begin{definition}\label{defn:pcontrolledint} Let    $m\geq 1$ an integer. A rectifiable   path $Y\colon [0,N]\to  L(V,W)$ is said to be a \emph{$m$-controlled integrand of $X$}  if there exist $m$ rectifiable paths $Y^{1}, \ldots, Y^{m}$ such that for all $l=0,1, \ldots, m$, one has $ Y^{l}\colon[0,T]\to L(V^{\otimes (l+1)},W)$ and it satisfies
\begin{equation}\label{eq_def_p_controlled}
dY^{l}_t=Y^{l+1}_t\,dX_t\,,
\end{equation}
where we use the notation $Y^0=Y$. 
\end{definition}

\begin{remark}
 The notion of $m$-controlled integrand of $X$ enables us to deal with more general paths in that we can replace  the expression $Y_t=f(X_t)$ arising in the EML formula \eqref{eq:EML_main_formula} with $f\in C^m(V,L(V,W))$  a  $m$-controlled integrand of $X$ by  the more general expression  $Y_t^l=D^{l}f(X_t)$. 
\end{remark}

\begin{example} A typical class of $m$-controlled integrands is given by controlled differential equation driven by rectifiable processes. Given a triple $U, V, W$  of Banach spaces, every path of the form
$Y_t=f(\tilde{X}_t)$ with $f\colon U\to L(V,W)$ of class $C^m$  and $\tilde{X}_t\in U$ solution of a controlled differential equation
\[d\tilde{X}_t= g(\tilde{X}_t)\diff X_t\]
for some smooth function $g\colon U\to L(V, U)$,  is a  $p$-controlled integrand up to an initial datum. The computation of the functions  $Y^{1}, \ldots, Y^{m}$ in this setting will involve Faà di Bruno's formula in the context of Fréchet calculus.
\end{example}
The Stieltjes  integral
 \[\int_0^NY_t \diff X_t\]
 is  well-posed for any  $p$-controlled integrand of  $X$.  As before, we also introduce the backward  and forward Riemann-Stieltjes sums of $p$-controlled integrands
	\[I^{-}_N(Y, X)= \sum_{k = 0}^{N-1} Y_k \delta X_k  \,,\quad I^{+}_N(Y, X)=\sum_{k = 0}^{N-1} Y_{k+1}\delta X_k\,. \]
where as before, we use  the notation $\delta X_k := X_{k+1} -X_k$. Combining the notions of controlled integrand with the sawtooth signature, we obtain the first part of Theorem \ref{main_thm_1} without the the choice of the optimal constant.
 
	\begin{theorem} \label{thm:prelEML}[Preliminary ELM formula]
 Given a rectifiable path $X$  and a path $Y$ which is a $m$-controlled integrand for $X$ for some integer $m\geq 1$, we have the  following identity for any $b\in T_1^m(V)$ 
\begin{equation}\label{eq:pre_EML_main_formula}
I_N^{\pm} (Y, X) = \int_0^NY_t\diff X_t-[Y_t\pi_1(b)]_{t=0}^N +\sum_{l = 2}^{m}(- 1)^{l}[Y^{l-1}_t \pi_l(Z^{\pm}_t(X,b))]_{t=0}^N+ R^{m}(b)\,,
\end{equation}
where the remainder $R^{m}(b)$ is given by the Stieltjes integral 
\begin{equation}\label{eq:rem_pre_EML}
    R^{m}(b)=(- 1)^{m+1}\int_{0}^{N}  Y^{m}_t  \mathrm{d} \pi_{m+1}(Z^{\pm}_t(X,b))\,.
\end{equation}
\end{theorem}
\begin{proof}
We first extend $Y$ and $X$ to $(-1, N+1)$  by their constant value at $0$ and $N$ (Any other extension would be effective for the result but we use this one to be more explicit). We  focus  on  $Z^{+}$ for the sake of conciseness and  closely follow the steps of the proof by Boas recalled above. The elementary identity
\begin{equation}\label{eq_discrete_to_continuous}
I_N^{+} (Y, X)= \int_{1-a}^{N+1-a} \,Y_t \,\mathrm{d}  X_{[t]}\,,
\end{equation}
for any $a\in (0,1)$, leads to the following expression of the difference
\[ D_N^a  =I_N^{+} (Y, X)-\int_{1 - a}^{N +1- a} Y_t\,\mathrm{d}  X_t   =-\int_{1-a}^{N+1-a} Y_t \,\mathrm{d}  (X_t- X_{[t]}+b^1)\,.\]
Using the first condition in \eqref{def_saw_signature1} on $\pi_1(Z^{+}_t(X, b))$ and the properties of $Y$, we  apply a first integration by part formula:
\[  D_N^a  = -[Y_t \pi_1(Z^{+}_t(X, b))]^{N+1-a}_{1-a}+ \int_{1 - a}^{N+1 - a} \mathrm{d}  Y_t \pi_1(Z^{+}_t(X, b)) \,.\]
Combining the property of a $p$-controlled integrand in  \eqref{eq_def_p_controlled} with the recursive formulae in \eqref{def_saw_signature2} yields \[\int_{1 - a}^{N+1 - a} \mathrm{d} Y_t \pi_1(Z^{+}_t(X, b))=\int_{1 - a}^{N+1 - a} Y^1_t \mathrm{d} \pi_2(Z^{+}_t(X, b))\,.\]
Iterating the integration by parts formula up to order $m$ leads to
\[  D_N^a    = \sum_{l = 1}^{m}(- 1)^{l}[Y^{l-1}_t\pi_l (Z^{+}_t(X, b)) ]^{N+1-a}_{t=1-a} + (- 1)^{m+1} \int_{1 - a}^{N+1 - a}  Y^{m}_t \diff X_t\otimes(\pi_m(Z^{+}_t(X, b))\,.\]
Using the continuity of the paths $ Y^{l}$ and $\pi_l(Z^{+}_t(X, b))$ for $l=2,\,\ldots \,,m$ and the definition of $\pi_1(Z^{+}_t(X, b))$, we can let  $a $ tend to $1$ obtaining \eqref{eq:pre_EML_main_formula} in its forward version.

The  formula \eqref{eq:pre_EML_main_formula} in its backward version can be derived in a similar manner from the shifted version of the identity \eqref{eq_discrete_to_continuous}, i.e. 
\begin{equation}\label{eq_discrete_to_continuous_2}
\sum_{k = 0}^{N - 1} Y_{k} \delta X_{k}= \int_{-a}^{N-a} Y_t \mathrm{d}  X_{[t+1]}\,.
\end{equation}
We conclude in a similar manner replacing the integral $\int_{1-a}^{N+1-a}$ by the integral $\int_{-a}^{N-a}$ and letting $a$ tend to $ 0$.
\end{proof}

\section{Optimal choice of initial datum and EML formula}

Let  as before, $X$  be a rectifiable process  and  $Y$ a process which is a $m$-controlled integrand for $X$ for some integer $m\geq 1$. In order to derive \eqref{eq:EML_main_formula} from   \eqref{eq:pre_EML_main_formula}, we  impose an optimality criterion so as to choose the initial constant $b$ independently  of the underlying integrand $Y$ and minimise the expected remainder $\mathbb{E}\Vert R^{m}(b) \Vert$ with $R^{m}(b) $ given in \eqref{eq:pre_EML_main_formula}.  Combining the standard  inequality \[\Vert A x\Vert_W\leq \Vert A\Vert_{ L(V^{\otimes k}, W)} \Vert x\Vert_{ V^{\otimes k}}\] 
for any $A\in L(V^{\otimes k},W)$ and $x\in V^{\otimes k} $ with the continuity of the admissible norm, for any $b\in T_1^m(V)$ we obtain the following  estimate on the expected remainder in \eqref{eq:rem_pre_EML}:
 \[\mathbb{E}\Vert R^{m}(b) \Vert \leq   \mathbb{E}\int_{0}^{N}  \Vert Y^{m}_t \Vert_{ L(V^{\otimes (m+1)}, W)}   \Vert \pi_{m}(Z^{\pm}_t(X,b)) \Vert_{V^{\otimes m}} |\mathrm{d} X_t|\,,\]
where $|\mathrm{d} X_t| $ is the total variation measure of the path $X$. To disentangle the dependence on $b$ from $Y$ in a symmetric way, we   apply the usual Cauchy–Schwarz inequality to the product measure $\mathbb{P}\otimes |\mathrm{d}X_t|$ obtaining 
  \[\mathbb{E}\Vert R^{m}(b) \Vert \leq   \left(\mathbb{E}\int_{0}^{N}  \Vert Y^{m}_t \Vert_{ L(V^{\otimes (m+1)}, W)}^2 |\mathrm{d} X_t| \right)^{1/2} \left(\mathbb{E}\int_{0}^{N}  \Vert \pi_{m}(Z^{\pm}_t(X,b)) \Vert_{V^{\otimes m}}^2 |\mathrm{d} X_t|\right)^{1/2}.\]
 Minimising the  $L^2$-norm of the second factor in the right-hand side seems a reasonable strategy to fix $b$. Unfortunately,
 when we want to minimize for any  $m\geq 1$ the functional
 \begin{equation}\label{eq:false_equation}
     b\in T^m_1(V)\to\bE\int_0^N \Vert \pi_m(Z^{\pm}_t(X, b))\Vert^2_{V^{\otimes l}}|\mathrm{d} X_t|\,
 \end{equation}
 this choice depends   on the index $m$ in essential way and it does not guarantee a recursive definition of a sequence of constants $b\in T_1((V))$, meaning that for any couple of integers $m\geq n\geq 1$ and $b^n$, $b^m$ minimal value of \eqref{eq:false_equation} the the projection on the lower $n$ components of $b^m$ does not necessarily coincide with $b^n$. 
 
In  alternative to this approach, we want to define a optimisation procedure such that for any given optimal element of $b\in T^m_1(V)$ there exists a unique element $v\in V^{\otimes (m+1)}$ that minimises  the variance only on the new components. In the special setting of Hilbert spaces, this optimisation problem has an explicit solution.
 
  \begin{theorem}\label{prop_optimisation_P} 
Let $V$ be a Hilbert space. For any given  $b\in T^m_1(V)$  the functional 
\[
v\in V^{\otimes (m+1)} \to\bE\int_0^N \Vert \pi_{m+1}(Z^{\pm}_t(X, b+ v))\Vert^2_{V^{\otimes (m+1)}}|\mathrm{d} X_t|\,
\]  
has a unique minimum at the value 
 \begin{equation}\label{eq:optimal_constants}
b^* = - \frac{1}{\mathbb{E}\int_0^N |\mathrm{d} X_t|} \mathbb{E}\int_0^N  \pi_{m+1}(Z^{\pm}_t(X,(b,0))  |\mathrm{d} X_t|\,,
 \end{equation}
 where $(b,0)\in T^{m+1}_1(V)$ is the trivial extension of  $b$ to $T^{m+1}_1(V)$ by setting  $\pi_{m+1}(b,0)=0$.
\end{theorem}

\begin{proof}[Proof of Theorem \ref{prop_optimisation_P}]
It is sufficient to prove the theorem for the forward sawtooth signature. Writing the sum  $b+v$ as $(b,0)+v $, we use the identity  \eqref{proof_saw_signature_1} in the proof of Theorem \ref{thm_saw_signature_1} to  write
\begin{align*}
& \bE\int_0^N \Vert \pi_{m+1}(Z^{+}_t(X, (b,0)+ v))\Vert^2_{V^{\otimes {m+1}}}|\mathrm{d} X_t| \\&=\bE\int_0^N \Vert  \pi_{m+1} (S^{\flat}_{0,t}(X)\otimes ((b,0)+ v))-\sum_{k=0}^{[t]-1}\pi_{m}({S}^{\flat}_{k+1,t} (X)\otimes\delta X_k)\Vert^2_{V^{\otimes {m+1}}}|\mathrm{d} X_t|   \\&=\bE\int_0^N \Vert  \pi_{m+1} (S^{\flat}_{0,t}(X)\otimes (b,0))-\sum_{k=0}^{[t]-1}\pi_{m}({S}^{\flat}_{k+1,t} (X)\otimes\delta X_k)+ v\Vert^2_{V^{\otimes {m+1}}}|\mathrm{d} X_t|  \\&= \bE\int_0^N \Vert \pi_{m+1}(Z^{+}_t(X, (b,0)))+  v\Vert^2_{V^{\otimes {m+1}}}|\mathrm{d} X_t| \,.
		\end{align*} 
The optimality property of the value \eqref{eq:optimal_constants} then  follows from  averages in quadratic optimisation, which we recall   here for sake of completeness. For any fixed $v\in V^{\otimes {m+1}}$ we  write 
\begin{align*}
&\bE\int_0^N \Vert \pi_{m+1}(Z^{+}_t(X, (b,0)))+  v\Vert^2_{V^{\otimes {m+1}}}|\mathrm{d} X_t| \\&
=\bE\int_0^N \Vert  \pi_{m+1}(Z^{+}_t(X, (b,0)))+b^*\Vert^2_{V^{\otimes {m+1}}}|\mathrm{d} X_t|+\bE\int_0^N \Vert v-b^*\Vert^2_{V^{\otimes {m+1}}}|\mathrm{d} X_t| \\&+ 2 \bE\int_0^N \langle  \pi_{m+1}(Z^{+}_t(X, (b,0))) +b^*,v-  b^*\rangle|\mathrm{d} X_t|\,,
\end{align*}
where  $\langle\,, \rangle $ is the canonical extension to $V^{\otimes (m+1)}$ of the Hilbert bracket on $V$. By definition of $b^*$, see \eqref{eq:optimal_constants}, one has \begin{align*}
&\bE\int_0^N \langle \pi_{m+1}(Z^{+}_t(X, (b,0))) +b^*,v-  b^*\rangle|\mathrm{d} X_t|\\&= \langle\bE\int_0^N \pi_{m+1}(Z^{+}_t(X, (b,0)))+b^*|\mathrm{d} X_t|,v-  b^*\rangle= 0\,.
\end{align*}
Therefore one has
\begin{align*}
&\bE\int_0^N \Vert \pi_{m+1}(Z^{+}_t(X, (b,0)))+  v\Vert^2_{V^{\otimes (m+1)}}|\mathrm{d} X_t|\\&\geq \bE\int_0^N \Vert \pi_{m+1}(Z^{+}_t(X, (b,0))) +b^*\Vert^2_{V^{\otimes (m+1)}}|\mathrm{d} X_t|
\end{align*}
and $b^*$ is the unique global minimum thanks to the strong convexity of the norm function $\Vert \cdot \Vert_{V^{\otimes (m+1)}}$.
\end{proof}
From this expression we to define the optimal forward (backward) optimal tensor and sawtooth signatures.
\begin{definition}\label{defn_optimal}
Let  $X$  be a rectifiable process, we call $b^{\pm} \in T_1((V))$ the \emph{optimal forward and backward  tensors}  the tensor series defined in the following recursive way: one set $\pi_0(b^{\pm})=1$, then 
\begin{equation}\label{eq:optimal_constants_def_1}
 \pi_1(b^{\pm}) =  \frac{1}{\mathbb{E}\int_0^N |\mathrm{d} X_t|} \mathbb{E}\int_0^N  \pi_1(Z^{\pm}_t(X,\mathbf{1}))  |\mathrm{d} X_t|\,,
 \end{equation}
and for any integer $l\geq 2$ denoting by $b^{\pm,< l}=(b_0^{\pm}\,,-b_1^{\pm} \ldots\,, b_{l-1}^{\pm}, 0, 0, \ldots )$ the sequence obtained by truncating to zero all the previous term of $b^{\pm}$ before $l$ and taking the opposite value at level $1$, the term  $\pi_l(b^{\pm})$ is obtained by setting  \begin{equation}\label{eq:optimal_constants_def}
 \pi_l(b^{\pm}) = - \frac{1}{\mathbb{E}\int_0^N |\mathrm{d} X_t|} \mathbb{E}\int_0^N  \pi_l(Z^{\pm}_t(X,b^{\pm,<l}))  |\mathrm{d} X_t|\,.
 \end{equation}
    The corresponding sawtooth signature   $Z^{\pm}_t(X, \bar{b}^{\pm})$ with $\pi_1(\bar{b}^{\pm})=- \pi_1(b^{\pm}) $ is called the \emph{optimal forward} and \emph{backward sawtooth signature}.
 \end{definition}
\begin{remark}\label{rk_minus_sign}
    The condition \eqref{eq:optimal_constants_def_1} differs from \eqref{eq:optimal_constants_def} by a sign factor which is recorrected in the term $b^{\pm,<l}$ and $\bar{b}^{\pm}$ because we want to rewrite the factor $-\pi_1(b)$ appearing when $m=1$ in \eqref{eq:pre_EML_main_formula} as a sum, so that the identity \eqref{eq:EML_main_formula} looks more similar to \eqref{eq:clas_EML}, where no negative sign is employed. By a trivial change of variable this term has the simple interpretation of the minimum of the functional 
    \[
v\in V \to\bE\int_0^N \Vert \pi_{1}(Z^{\pm}_t(X, \mathbf{1}- v))\Vert^2_{V}|\mathrm{d} X_t|\,.
\]  
\end{remark}

We are now ready to show that the optimal forward and backward tensors extend the Bernoulli numbers. This result can be obtained as consequence of a more general recursion relation.

\begin{proposition}
Given  a positive non-zero random variable   $\alpha$  with finite moments of all orders, the optimal average tensors associated to the case $V=\bR$ and   $X_t=\alpha t $  are given at level $1$ by 
\begin{align} \label{recursion_forward_1}
\pi_1(b^{+})&=   \frac{\mathbb{E}[\alpha^2]}{\mathbb{E}[\alpha] } B^{+}_1   \,,\quad \pi_1(b^{-})=   \frac{\mathbb{E}[\alpha^2]}{\mathbb{E}[\alpha] } B^{-}_1   \,.
\end{align}
Moreover, for any integer $l\geq 2$ one has the following recursions 
  		\begin{align} \nonumber
\pi_l(b^{+})&=   \frac{\mathbb{E}[\alpha^{l+1}]}{\mathbb{E}[\alpha] }\frac{B^{+}_l}{l!} + \left(\mathbb{E}[\alpha^{l}]\frac{\mathbb{E}[\alpha^{2}]}{\mathbb{E}[\alpha]^2}-\frac{\mathbb{E}[\alpha^{l+1}]}{\mathbb{E}[\alpha] }\right)\frac{N^{l-1}}{l!}B^{+}_1\\&+ \sum_{j=2}^{l-1}\frac{N^{l-j}}{(l+1-j)!}\bigg(\frac{\mathbb{E}[\alpha^{l+1}]}{\mathbb{E}[\alpha] }\frac{B^{+}_j}{j!}-\frac{\mathbb{E}[\alpha^{l+1-j}]}{\mathbb{E}[\alpha] }\pi_j(b^+) \bigg) \,,\label{recursion_forward}\\\nonumber
  \pi_l(b^{-})&= \frac{\mathbb{E}[\alpha^{l+1}]}{\mathbb{E}[\alpha] }\frac{B^{-}_l}{l!} + \left(\mathbb{E}[\alpha^{l}]\frac{\mathbb{E}[\alpha^{2}]}{\mathbb{E}[\alpha]^2}-\frac{\mathbb{E}[\alpha^{l+1}]}{\mathbb{E}[\alpha] }\right)\frac{N^{l-1}}{l!}B^{-}_1\\&+ \sum_{j=2}^{l-1}\frac{N^{l-j}}{(l+1-j)!}\bigg(\frac{\mathbb{E}[\alpha^{l+1}]}{\mathbb{E}[\alpha] }\frac{B^{-}_j}{j!}-\frac{\mathbb{E}[\alpha^{l+1-j}]}{\mathbb{E}[\alpha] }\pi_j(b^-) \bigg) \,.
  \label{recursion_backward}
   \end{align}
\end{proposition}
\begin{proof} 
 By assumption on $\alpha$, the expressions \eqref{recursion_forward} and \eqref{recursion_backward} are well defined. We first consider  the case of forward optimal average tensors. Using the a.s. relation 
		\[ \alpha\int_0^N \pi_{l}(Z^{+}_s(\alpha t, b^{+,<l}))ds=\pi_{l+1}(Z^{+}_N(\alpha t, b^{+,<l}))\]
 and the identities \eqref{Bernoulli_expl_saw_sign1}, \eqref{Bernoulli_expl_saw_sign2}   one has   when $l=1$ the a.s. identity  
    \begin{align*}
&\frac{\alpha }{N\mathbb{E}[\alpha] } \int_0^N \pi_1(Z^{\pm}_s(\alpha t,\mathbf{1})) ds = \frac{1}{N\mathbb{E}[\alpha] } \pi_{2}(Z^{\pm}_N(\alpha t,\mathbf{1}))=\frac{1}{N\mathbb{E}[\alpha] }- \alpha^2N B^{\mp}_1
\end{align*}
 which becomes  \eqref{recursion_forward_1} when we take the expectation and we check the definition in \eqref{eq:optimal_constants_def_1}. for any $l\geq 
 2$ we still use \eqref{Bernoulli_expl_saw_sign1} and we derive in the forward case
  \begin{align*}
  &- \frac{\alpha }{N\mathbb{E}[\alpha] } \int_0^N \pi_l(Z^{+}_s(\alpha t,b^{+,<l})) ds =- \frac{1}{N\mathbb{E}[\alpha] } \pi_{l+1}(Z^{+}_N(\alpha t,b^{+,<l}))\\&= \frac{\alpha^{l+1}}{\mathbb{E}[\alpha] }\sum_{j=2}^{l}\frac{N^{l-j}}{(l+1-j)!}\frac{B^{-}_j}{j!}-\sum_{j=2}^{l-1}\frac{\alpha^{l+1-j}N^{l-j}}{\mathbb{E}[\alpha] (l+1-j)!}\pi_j(b^+)  + \frac{N^{l}}{l!}(\frac{\alpha^{l+1}}{\mathbb{E}[\alpha] }B^{-}_1+\alpha^{l} \frac{\mathbb{E}[\alpha^2]}{\mathbb{E}[\alpha] } B^{+}_1) \,.
  \end{align*}
  By taking the expectation and using the key-identities on Bernoulli numbers
  \begin{equation}\label{bern_prop}
  B^{-}_1=-B^{+}_1\quad B^{-}_l=B^{+}_l\quad \text{for any integer $l\geq 2$}
  \end{equation}
  we obtain   \eqref{recursion_forward}.  To obtain \eqref{recursion_backward}, we repeat the  same calculations  starting from \eqref{Bernoulli_expl_saw_sign2}.
 %  \begin{align*}
 % &- \frac{\alpha }{N\mathbb{E}[\alpha] } \int_0^N \pi_l(Z^{-}_s(\alpha t,b^{-,<l})) ds =- \frac{1}{N\mathbb{E}[\alpha] } \pi_{l+1}(Z^{-}_N(\alpha t,b^{-,<l}))\\&= \frac{1}{N\mathbb{E}[\alpha] }\left(\alpha^{l+1}\sum_{j=0}^l\frac{N^{l+1-j}}{(l+1-j)!}\frac{B^{+}_j}{j!} -\sum_{j=0}^{l-1}\frac{\alpha^{l+1-j}N^{l+1-j}}{(l+1-j)!}\pi_j(b^-)   \right) \,, 
 % \end{align*}
\end{proof}

\begin{theorem}\label{thm:classical_bernoulli}
When   $V=\bR$ and $X_t=\lambda t$  with $\lambda>0$ one has 
\begin{equation}\label{eq:opt_constants1}
\pi_l(b^{\pm})=\lambda^l\frac{B^{\pm}_l}{l!}\, , \quad \pi_m(Z^{\pm}_N(\lambda t,\bar{b}^{\pm}))=\lambda^m\frac{B^{\pm}_m}{m!} 
\end{equation}
for any couple of integers $l\geq 1$ and $m\geq 2$. If $\lambda=1$, we further have 
\begin{equation} \label{eq:remainder}
 \pi_l(Z^{\pm}_t( t,\bar{b}^{\pm}))=\frac{1}{l!}P_l(t-[t])
\end{equation}
for any $l\geq 2$ and the formula
\eqref{eq:EML_main_formula} coincides with \eqref{eq:clas_EML}.
\end{theorem}
\begin{proof}
We will prove the first identity in \eqref{eq:opt_constants1} by induction on $l$. We first observe that   $\pi_0(b^{\pm})= B_0^{\pm}=1$. Moreover at level $1$, \eqref{eq:opt_constants1} is still true  from \eqref{recursion_forward_1}. Assuming the result  true  at any order strictly smaller than $l\geq 2$, we insert the expression $\pi_n(b^{\pm})=\lambda^n\frac{B^{\pm}_n}{n!}$ for any $n<l$ inside the recursions \eqref{recursion_forward} \eqref{recursion_backward} and  we easily obtain
\begin{align*}
\pi_l(b^{\pm})=\lambda^l\frac{B^{\pm}_l}{l!}\,.
\end{align*}
  To check the second part of \eqref{eq:opt_constants1}, we compute $ \pi_m(Z^{\pm}_N(\lambda t,\bar{b}^{\pm}))$ directly   from \eqref{Bernoulli_expl_saw_sign1} and \eqref{Bernoulli_expl_saw_sign2}. For instance in the forward setting we have 
 \begin{align*}
  \pi_m(Z^{+}_N(\lambda t,\bar{b}^{+})) &= \sum_{j=1}^{m}\frac{\lambda^{m-j}N^{m-j}}{(m-j)!}\pi_j(\bar{b}^{+})- \lambda ^m\sum_{j=1}^{m-1}\frac{N^{m-j}}{(m-j)!}\frac{B^{-}_j}{j!}=\lambda^m\frac{B^{\pm}_m}{m!} \,,
  \end{align*}
 as a consequence of the properties in \eqref{bern_prop}. A similar computation holds in the backward setting. If $\lambda=1$,  plugging the identities \eqref{eq:opt_constants1} into \eqref{eq:EML_main_formula} is yields \eqref{eq:clas_EML} modulo the remainder terms \eqref{eq:rem_EML} and \eqref{eq:clas_rem}, which a fortiori have to coincide. We therefore obtain
\[\int_0^N f^{(m)}(s)\pi_m(Z^{\pm}_s( t,\bar{b}^{\pm}))\diff s= \frac{1}{ m!}\int_0^Nf^{(m)}(s)P_m(s-[s])\diff s\]
for any smooth function $f$, which yields the desired result.
\end{proof}

\section{Interpolating discrete and continuous signature}

After exploring the choice of possible constants compatible with the classical EML formula, we  shall use the formula \eqref{eq:pre_EML_main_formula} in the backward case where $b=\mathbf{1}$  to deduce a new deterministic identity linking the signature of a rectifiable path $X$ with its discrete versions, defined from a finite set of values $x\colon [[0, N]]\to V$ with  $[[0, N]]=\{0\,, \ldots N\}$. We will refer to $x$ as a \emph{time series}. To link these two notions, we will always relate a rectifiable path $X\colon [0, N]\to V$ and a time-series $x$ such that $X$ interpolates $x$ i.e. $X_k=x_k$ for any $k\in [[0, N]]$.
Throughout this section we will also use the notation $Z\colon [0, N]\to T_1((V))$ to denote the backward sawtooth signature $Z^{-}(X)$.

\subsection*{Sawtooth sum signature and compositions}
Using these conventions, the formula \eqref{eq:pre_EML_main_formula} turns into a simpler identity.

\begin{proposition}
\label{thm:correction_signature}
For any  rectifiable process $X\colon [0,N]\to V$ interpolating the values of a time series  $x\colon [[0, N]]\to V$ one has  for any integer $l\geq 1$ any integer time $n\in [[0, N]]$
\begin{equation}\label{eq:sawtooth} \pi_l(S_{0,n}(X))= \sum_{k=0}^{n-1}\pi_{l-1}(S_{0,k}(X))\otimes \delta x_k  + \sum_{p= 2}^{l}(-1)^{p+1} \pi_{l-p}(S_{0,n}(X))\otimes  \pi_p(Z_n)\,.
\end{equation}
\end{proposition}
\begin{proof}
 To prove the identity we   apply the backward formulation of \eqref{eq:pre_EML_main_formula} 
where we choose the path $Y\colon [0, n]\to L(V, V^{\otimes l})$  given by $Y_t v= \pi_{l-1}(S_{0,t}(X))\otimes v $, the initial datum    $b=\mathbf{1}$
 and   $m=l+1$. Since $ R^{n+1}(b)=0$ for any $b$ and $I^{-}_N(Y, X)= \sum_{k = 0}^{n-1} Y_k \delta X_k $,    we conclude.
\end{proof}

This simple formula can be iterated to compare the components of the signature with the \emph{sum signature} of $x$, defined  as the  map $\Sigma(x)\colon [[0, N]]^2_<\to T_1((V))$
given on any component $l\geq 1$ by 
\[ \pi_l(\Sigma_{m,n}(x))=\sum_{m\leq i_1< \ldots i_l< n} \delta x_{i_1} \otimes\ldots \otimes \delta x_{i_l}\,.\]
 Clearly, for any path $X$ interpolating a time series $x$ one has the trivial identity 
\[\pi_1(S_{0,N}(X))= \pi_1(\Sigma_{0,N}(x))= x_N-x_0.\]
However, at higher level, such an identity  does not hold any longer due to the appearance of correction terms, which mix the values $x$ with the "shape" of the chosen interpolation. We represent them by means of the combinatorial tool of compositions. 
 
Given an integer  $n\geq 1,$ a \emph{composition} of $n$ is a vector $I=(i_1\,, \ldots ,i_l)$ whose   components are strictly positive integers that add up to $n$.   Let $C(n)$ denote the set of compositions of $n$. We call  the element of the set $C(1)\cup C(2)\cup \ldots $  a composition. For practical reason we also  add to the set of the composition an extra symbol $\emptyset $ which we call the empty composition. By construction on any given composition  $I=(i_1\,, \ldots ,i_l)$  (even when $I=\emptyset$ in  a trivial way) we can unambiguously define the values 
\[|I|= l\,,  \quad \Vert I \Vert= \sum_{k=1}^{|I|}i_k\,,\quad I!= i_1!\ldots i_l!. \] 
We call $\Vert I \Vert$, $|I|$, $I!$ respectively the sum,  the length and the factorial of a  composition $I$.

We want to extract information about the position of the integer $1$ among the components from the values of a composition $I=(i_1\,, \ldots ,i_l)$. Let us denote by $I^{\sharp}$ the longest vector  contained in $I$ having the value $1$ as its terminal value on the right,  which we call  the \emph{unity projection} of $I$. Similarly we denote by $I^*$ the vector obtained  by suppressing all the values of $1$ inside the components of $I$, which we call  the \emph{non-unity projection} of $I$. If the composition contains only coordinates with $1$ we simply set $I^*=\emptyset$ and similarly  if $I$ does not contains $1$ one has $I^{\sharp}=\emptyset$. Here are some examples
\[(1,1,1)^{\sharp}=(1,1,1)\,, \quad (1,1,1)^{*}=\emptyset \,,\quad (2)^\sharp=\emptyset\]
\[  (1,2)^{\sharp}= (1)\,, \quad (2,1)^{\sharp}=(2,1)\,, \quad (2,1)^{*}=(1,2)^*=(2)\,.\]
By simply iterating the identity \eqref{eq:sawtooth} we can easily derive an identity to express the correction term between $S(X)$ and $\Sigma(x)$ by coupling $\Sigma(x)$ and $Z$ in a bigger tensor structure. To properly define it, we introduce the space of \emph{composition tensor series}
\[T^C(V)=\mathbb{R}\mathbf{1}\oplus \prod_{n\geq 1} (V^{\otimes n})^{C(n)}\,.\]
Every element of $v\in T^C(V)$ will consists in a sequence of tensor $v=(v_0\,, v_1\,,\ldots)$  where each tensor $v_n$ will be  a cartesian product of elements in $V^{\otimes n}$ that are indexed by $C(n)$. Similarly to the case of standard tensor series, for any composition $I$ we denote by $\pi_I (v)\in V^{\otimes \Vert I\Vert}$ the projection of each element $v\in T^C(V)$ to the  $I$-th component inside $T^C(V)$ and we use the notation $T^C_1(V)$ for the tensor series whose first component is equal to $1$. The extra components contained in $T^C(V)$ will allow us to express new terms appearing recursively from \eqref{eq:sawtooth}.

\begin{definition}  
For any  rectifiable process $X\colon [0,N]\to V$ interpolating the values of a time series  $x\colon [[0, N]]\to V$ we define the two-parameter map $\Sigma(x, Z)\colon [[0, N]]^2_<\to T^C_1(V)$ whose components are given  for any composition $I=(i_1\,, \ldots ,i_l)$
\begin{equation}\label{eq:partial_it_sum_saw}
    \pi_I(\Sigma_{m,n}( x,Z))=\sum_{m\leq j_1<\ldots <j_{q}<n} \pi_{i_1}^* (\delta x_{j_1},Z_{j_1}) \otimes\ldots \otimes \delta x_{j_q}\otimes \pi_{i_{q+1}}(Z_n) \otimes  \ldots \otimes \pi_{i_{l}}(Z_n)\,,
\end{equation}  
 where    $|I^{\sharp}|=q$ 
 and we use the notation
\[\pi_{i_1}^* (\delta x_{j_1},Z_{j_1})=\begin{cases}
    \delta x_{j_1} & \text{if} \quad i_1=1\\\pi_{i_1} (Z_{j_1})& \text{if} \quad i_1>1
\end{cases}\,.\] 
We call the expression \eqref{eq:partial_it_sum_saw} the
\emph{sawtooth sum signature} of level $I$.
\end{definition}

The sawtooth sum signature  $   \Sigma_{m,n}( x,Z)$ provides a way to couple $\Sigma(x)$ and $Z$ by projecting on different components. Indeed  for any integer $l\geq 2$ one has 
\begin{equation}\label{eq:trivial_extension}
\pi_l(S_{m,n}(x))= \pi_{(1\,, \ldots \,, 1)_l} (\Sigma^{}_{m,n}( x,Z))\,\quad \pi_l(Z_n(x))= \pi_{(l)} (\Sigma_{m,n}( x,Z))\,,
\end{equation}
where we denote by $(1\,, \ldots \,, 1)_l$ and $(l)$ the composition on $l$ containing only $1$ among its components an the trivial composition given the integer $l$ itself. All the other coordinates involve mixtures of sums with the terms of $Z$ and they allow to compare  the signature $S(X)$ with its discrete version at any level.
\begin{theorem}\label{thm:correction_signature2}
For any  rectifiable process $X\colon [0,N]\to V$ interpolating the points of a time series  $x\colon [[0, N]]\to V$, one has  for any integer $l\geq 1$ and any integer time $n\in [[0, N]]$
\begin{equation}\label{eq:discrete_signature}
   \pi_l(S_{0,n}(X))=\pi_l(\Sigma_{0,n}(x))+ \sum_{I\in C(l), \; I \neq (1\,, \ldots \,, 1)_l}(-1)^{\Vert I^*\Vert + \vert I^*\vert} \pi_I(\Sigma_{0,n}(x,Z))\,. 
\end{equation}
\end{theorem}

\begin{proof}
We prove the result by induction on $l$. The  initial step of the induction follows trivially from the definition of $\Sigma(x)$ and $S(X)$. Supposing the identity true for some integer $l>1$, using \eqref{eq:trivial_extension} we write the induction hypothesis as 
\[\pi_l(S_{0,n}(X))= \sum_{I\in C(l)}(-1)^{\Vert I^*\Vert + \vert I^*\vert}  \pi_I(\Sigma_{0,n}(x,Z))\,.\]
To prove the induction step, we combine \eqref{eq:sawtooth} with the induction hypothesis, thereby obtaining    \begin{align*}
        \pi_{l+1}(S_{0,n}(X))&=\sum_{k=0}^{n-1}\pi_{l}(S_{0,k}(X))\otimes \delta x_k  + \sum_{p= 2}^{l+1}(-1)^{p+1} \pi_{l-p}(S_{0,n}(X))\otimes  \pi_p(Z_n)\\&=
        \sum_{I\in C(l)}(-1)^{\Vert I^*\Vert + \vert I^*\vert}\sum_{k = 0}^{n-1} \pi_I(\Sigma_{0,k}(x,Z))\otimes \delta x_k  \\&+ \sum_{p = 2}^{l+1}(-1)^{p+1}\sum_{J\in C(l+1-p)}(-1)^{\Vert J^*\Vert + \vert J^*\vert}  \pi_J(\Sigma_{0,n}(x,Z))\otimes \pi_{p}(Z_n) 
    \end{align*}
We then apply the definition \eqref{eq:partial_it_sum_saw} and the right-hand of the above sum becomes
    \begin{align*}
        &\sum_{I\in C(l)}(-1)^{\Vert (I,1)^*\Vert + \vert (I,1)^*\vert}  \pi_{(I,1)}(\Sigma_{0,n}(x,Z))\\&+ \sum_{p = 2}^{l+1}\sum_{J\in C(l+1-p)}(-1)^{\Vert (J,p)^*\Vert + \vert (J,p)^*\vert}  \pi_{(J,p)}(\Sigma_{0,n}(x,Z))\,.    \end{align*}
 The induction step then follows from the fact that every composition of $I\in C(l+1)$ is of the form $(I,1)$ for some $I\in C(l)$ or of the form $(J,p)$ for some $J\in C(l+1-p)$.
\end{proof}
\subsection*{Iterated sum signature of a time series}
The formula \eqref{eq:partial_it_sum_saw} cannot be more explicit because it depends on the chosen  interpolation, which is encoded in   the values of $Z$ at integer values. We conclude this section by showing how the recursive identity  \eqref{eq:sawtooth}in the special case of a finite dimensional vector space $V= \mathbb{R}^d $, allows to deduce a known formula between signature of linear interpolating path and the iterated sum signature, \cite[Definition 3.1]{Tapia20}.

For this purpose, we briefly recall this notion. We denote by $\mathfrak{A}_d$ the vector space generated by the free commutative and associative real semi-group over $d$ elements. An elementary way to realise it is to consider the vector space of polynomials $\mathbb{R}[X_1\,, \ldots\,, X_d]$ and then we consider the quotient vector space
\[\mathfrak{A}_d= \mathbb{R}[X_1\,, \ldots\,, X_d]/\{P\in \mathbb{R}[X_1\,, \ldots\,, X_d]\colon P=a\,, a\in \mathbb{R}\}\]
This vector space has a basis given on of non-constant monomials $X_{i_1}\ldots X_{i_n}$ which we represent by  $[i_1\ldots i_n]$ with  $i_1\,, \ldots,  i_n\in \{1\,, \ldots \,,d\}$  with possible repetitions of the values. Moreover, the intrinsic product operation among polynomial can pass to the quotient and it  will be denoted by $[,]\colon \mathfrak{A}_d\times \mathfrak{A}_d\to \mathfrak{A}_d$. Using this notations, we have the trivial identity $[i_1\ldots i_n]= [i_1[i_{n-1}i_n\ldots ]]$ independently on the order of the letters $i_1\,, \ldots,  i_n$.

Starting from the space $\mathfrak{A}_d$, we use it to define unitary tensors series over it as elements of $T_1((\mathfrak{A}_d))$. 
Even if $\mathfrak{A}_d$ is infinite-dimensional we use the countable basis $\mathcal{B}=\{ [i_1\ldots i_n]\colon i_1\,, \ldots,  i_n\in \{1, \ldots, d\}\} $ as a starting alphabet and we identify  $T_1((\mathfrak{A}_d))$ as the space of unitary formal series on words built from this set. Moreover for any $v\in T_1((\mathfrak{A}_d))$  and word $w$ built from $\mathcal{B}$ we denote by $\langle v, w\rangle $ the value of the term $w$ inside $v$, see e.g. \cite[Section 2]{bellingeri2023discrete} for further details.  Finally, we denote the components of a time series $x\colon [[0, N]]\to \mathbb{R}^d$ by $x_k^j $ with $j\in \{1, \ldots k\}$ for any $k\in [[0,N]]$ so that one has $x_k= (x_k^1\,, \ldots \,, x_k^d)$.
\begin{definition}\label{definizione_ISS}
For any time series  $x\colon [[0, N]]\to \mathbb{R}^d$, we define the iterated sum signature of $x$  as the two-parameter map $\mathcal{S}(x)\colon [[0, N]]^2_<\to T_1((\mathfrak{A}_d))$ defined on each component by
\begin{align}\label{def_ISS}
   \langle \mathcal{S}_{m,n}(x)\,, [i_1\ldots i_{j_1}]\ldots [i_{j_{n-1}-1}\ldots i_{j_k}]\rangle=\sum_{m\leq l_1<\ldots <l_{k}<n} \delta x_{l_1}^{[i_1\ldots i_{j_1}]} \ldots \delta x_{l_k}^{[i_{j_{k-1}-1}\ldots i_{j_k}]}  \,,  
\end{align}
where we have set $\delta x_{l_n}^{[i_{1}\ldots i_{k}]}= \delta x_{l_n}^{i_1}\ldots  \delta x_{l_n}^{i_{k}}$.
\end{definition}
\begin{remark}
The above definition of iterated sum signature is  taken from
\cite{Tapia20}, but differs from the original version by a different  parametrization of the increment (that is $\delta x_{k}= x_k- x_{k-1}$). We follow  the convention on increments contained   \cite{bellingeri2023discrete} which is coherent with our previous notations.
\end{remark}
The presence of new letters obtained by contracting the original $d$ values, allows to define a natural operation between words and composition in setting for any $I \in C(n)$, $I=(i_1,\ldots, i_k)$
\[[j_1\ldots j_n]_I= [j_1\ldots j_{i_1}]\ldots [j_{i_{k-1}}\ldots j_n]\,.\]
With these notations we can actually explicitly compute the backward sawtooth signature of a path that is interpolating linearly the values of a time series.

\begin{proposition}
 For any given time series  $x\colon [[0, N]]\to \mathbb{R}^d$, 
 the backward sawtooth signature of its linear interpolation path $X$ is given for any component $j_1\ldots j_n$ with $n\geq 2$ by the expression
\begin{equation}\label{explicit_sawtooth}\begin{split}\langle Z_t, &j_1\ldots j_n \rangle =\left(\frac{(t-[t])^n}{n!}- \frac{(t-[t])^{n-1}}{(n-1)!}\right)\delta{x}_{[t]}^{[j_1\ldots j_n]} \\&+ \sum_{i=0}^{n-2}\frac{(t-[t])^i}{i!}\sum_{\substack{I\in C(n-i)\\i_1>1}} \left(\frac{1}{I!}-\frac{1}{(I-e_1)!}\right)\delta{x}_{[t]}^{^{[j_1\ldots j_i]}} \langle \mathcal{S}_{0,[t]}(x),[j_{n}\ldots j_{i+1}]_I \rangle \,,
\end{split} 
\end{equation}
where we use the notation $I-e_1$ to denote the composition $I$ obtained by subtracting $1$ to the first component of $I$ and eliminating it when $i_1=1$. 
\end{proposition}
 \begin{proof}
 The identity follows by induction on $n$. When  $n=2$ we use the explicit expression of $\pi_1(Z_t)$  given in  \eqref{def_saw_signature1} to derive the identity 
\begin{align*}
&\langle Z_t, j_1j_2\rangle=\int_{[t]}^{t} \delta{x}_{[t]}^{j_1} (X_s^{j_2}- x_{[t]+1}^{j_2}) \mathrm{d} s+\sum_{k=0}^{[t]-1}\int_{k}^{k+1} \delta{x}_k^{j_1} (X_s^{j_2}- x_{k+1}^{j_2}) \mathrm{d} s =\\&
= \left(\frac{(t-[t])^2}{2}-  (t-[t])\right)\delta{x}_{[t]}^{j_1}\delta{x}_{[t]}^{j_2}+\sum_{k=0}^{[t]-1}\delta{x}_k^{j_1}\delta{x}_k^{j_2}\int_{k}^{k+1}  (s- k-1) \mathrm{d} s \\&
=\left(\frac{(t-[t])^2}{2}-  (t-[t])\right)\delta{x}_{[t]}^{[j_1j_2]}-\frac{1}{2}\sum_{k=0}^{[t]-1}\delta{x}_{k}^{[j_1j_2]}\,.  
\end{align*}
  For the induction step we write for any $n>2$
\[\begin{split}
    &\langle Z_t, j_1\ldots j_{n+1}\rangle= \int_0^t   \langle Z_s, j_2\ldots j_{n+1}\rangle dX_s^{j_1} \\ 
    &=\sum_{k=0}^{[t]-1}\int_k^{k+1}\langle Z_s, j_2\ldots j_{n+1}\rangle \delta x_k^{j_1} ds +\delta x_{[t]}^{j_1}\int_{[t]}^t\langle Z_s, j_2\ldots j_{n+1}\rangle  ds= (I)+ (II)\,.
\end{split} \]
Let us consider the terms separately applying the induction hypothesis in both cases to $\langle Z_s, j_2\ldots j_n\rangle$. The first term reads
\[\begin{split}
(I)&=\sum_{k=0}^{[t]-1}\int_k^{k+1}\left(\frac{(s-k)^{n}}{n!}- \frac{(s-k)^{n-1}}{(n-1)!}\right)\delta{x}_{k}^{[j_1\ldots j_{n+1}]}\\&+ \sum_{k=0}^{[t]-1}\sum_{i=0}^{n-2}\int_k^{k+1}\frac{(s-k)^i}{i!}\sum_{\substack{I\in C(n-i)\\i_1>1}} \left(\frac{1}{I!}-\frac{1}{(I-e_1)!}\right) \langle \mathcal{S}_{0,k}(x),[j_{n+1}\ldots j_{i+1}]_I \rangle \delta{x}_{k}^{^{[j_1\ldots j_i]}}\,\\&=\sum_{k=0}^{[t]-1}\left(\frac{1}{(n+1)!}- \frac{1}{n!}\right)\delta{x}_{k}^{[j_1\ldots j_{n+1}]} \\&+\sum_{k=0}^{[t]-1}\sum_{i=0}^{n-2}\frac{1}{(i+1)!}\sum_{\substack{I\in C(n-i)\\i_1>1}} \left(\frac{1}{I!}-\frac{1}{(I-e_1)!}\right)\langle \mathcal{S}_{0,k}(x),[j_{n+1}\ldots j_{i+1}]_I \rangle\delta{x}_{k}^{^{[j_1\ldots j_i]}}\\ &=\sum_{I\in C(n+1)\,,\; \,i_1>1} \left(\frac{1}{I!}-\frac{1}{(I-e_1)!}\right) \langle \mathcal{S}_{0,[t]}(x),[j_{n}\ldots j_{1}]_I \rangle\,,
\end{split}\]
where the last equality is derived from the fact that any element $I \in C(n+1)$ with $i_1>1$ is $I=(n+1)$ or $I=(J, i)$ with $J \in C(n-j)$ and $j_1>1$. Dealing with the other expression in a similar manner we obtain
\[\begin{split}
(II)&=\int_{[t]}^{t}\delta x_{[t]}^{j_1}\left(\frac{(s-[t])^{n}}{n!}- \frac{(s-[t])^{n-1}}{(n-1)!}\right)\delta{x}_{[t]}^{[j_2\ldots j_n]}\mathrm{d} s\\&+ \sum_{i=0}^{n-2}\sum_{\substack{I\in C(n-i)\\i_1>1}} \left(\frac{1}{I!}-\frac{1}{(I-e_1)!}\right)\langle \mathcal{S}_{0,[t]}(x),[j_{n+1}\ldots j_{i+1}]_I \rangle \delta{x}_{[t]}^{^{[j_1\ldots j_i]}}\int_{[t]}^{t}\frac{(s-[t])^i}{i!} \mathrm{d} s\,\\&=\left(\frac{(t-[t])^{n+1}}{(n+1)!}-\frac{(t-[t])^{n}}{n!}\right)\delta{x}_{[t]}^{[j_1\ldots j_n]} \\&+\sum_{i=1}^{n-1}\frac{(t-[t])^i}{i!}\sum_{\substack{I\in C(n+1-i)\\i_1>1}} \left(\frac{1}{I!}-\frac{1}{(I-e_1)!}\right)\delta{x}_{[t]}^{^{[j_1\ldots j_i]}} \langle \mathcal{S}_{0,[t]}(x),[j_{n+1}\ldots j_{i+1}]_I \rangle\,.
\end{split}\]
 Combining the two expressions $(I)$ and $(II)$ yields the conclusion.
 \end{proof}
Evaluating the expression \eqref{explicit_sawtooth}  on integers and using the definition of $Z$ leads to simplified expressions of the flip and sawtooth signature on integer values.
 \begin{proposition} For any given time series $x\colon [[0, N]]\to \mathbb{R}^d$, 
 the backward sawtooth signature of the linear interpolation path $X$ on the integer values $m\in [[0, N]]$ is given for any component $j_1\ldots j_n$ with $n\geq 2$ by the expression
\begin{equation}\label{first_identity}
\begin{split}
     &\langle Z_m, j_1\ldots j_n\rangle =\sum_{\substack{I\in C(n)\\i_1>1}} \left(\frac{1}{I!}-\frac{1}{(I-e_1)!}\right)\langle \mathcal{S}_{0,m}(x),[j_{n}\ldots j_1]_I \rangle\\&=\sum_{\substack{I\in C(n)}} \frac{1}{I!}\langle \mathcal{S}_{0,m}(x),[j_{n}\ldots j_1]_I \rangle- \sum_{k=0}^{m-1}\sum_{\substack{I\in C(n-1)}} \frac{1}{I!}\langle\mathcal{S}_{k,m}(x),[j_{n-1}\ldots j_1]_I \rangle \delta x_k^{j_{n}}\,.
 \end{split}
 \end{equation}
 \end{proposition}
 \begin{proof}
The first part of \eqref{first_identity}   trivially follows from  \eqref{explicit_sawtooth}. The second equality follows from adding and subtracting the term 
\[\sum_{\substack{I\in C(n-1)}} \frac{1}{I!}\langle \mathcal{S}_{0,m}(x),j_{n}[j_{n-1}\ldots j_1]_I \rangle\,.\]
Then we obtain 
\[
\begin{split}
     &\langle Z_m, j_1\ldots j_n\rangle =\sum_{\substack{I\in C(n)}} \frac{1}{I!}\langle \mathcal{S}_{0,m}(x),[j_{n}\ldots j_1]_I \rangle- \sum_{\substack{I\in C(n-1)}} \frac{1}{I!}\langle\mathcal{S}_{0,m}(x),[j_nj_{n-1}\ldots j_1]_{I+e_1} \rangle \,.
 \end{split}
\]
Then we simply apply the Definition \ref{definizione_ISS} and develop the term $\langle\mathcal{S}_{0,m}(x),[j_nj_{n-1}\ldots j_1]_{I+e_1} \rangle$ 
according to the last index in the iterated sum to obtain
\[\sum_{\substack{I\in C(n-1)}} \frac{1}{I!}\langle\mathcal{S}_{0,m}(x),[j_nj_{n-1}\ldots j_1]_{I+e_1} \rangle=\sum_{k=0}^{m-1}\sum_{\substack{I\in C(n-1)}} \frac{1}{I!}\langle\mathcal{S}_{k,m}(x),[j_{n-1}\ldots j_1]_I \rangle \delta x_k^{j_{n}}\,.\]
\end{proof}
We finally use these preparatory results to reprove  a well-known equality in literature of  signature using our formulation of the EML identity in  Proposition \ref{thm:correction_signature} and some classical properties of the iterated sum signatures.
\begin{theorem}\label{last_thm}
For any given time series $x\colon [[0, N]]\to \mathbb{R}^d$, 
 the  signature of the linear interpolation path $X$ at the integer values $m\in [[0, N]]$ is given for any component $j_1\ldots j_n$ by the expression
\[ \langle S_{0,m}(X), j_1\ldots j_n\rangle= \sum_{I\in C(n)}\frac{1}{I!} \langle \mathcal{S}_{0,m}(x),[j_{1}\ldots j_{n}]_I \rangle.\]
\end{theorem}
 
	\begin{proof}
We prove the result by induction on $n\geq 1$ whose base is trivial. Using the identity  \ref{eq:sawtooth} projected on the component $j_1\ldots j_{n+1}$ and \eqref{first_identity} we derive
\[\begin{split}
   & \langle S_{0,m}(X), j_1\ldots j_{n+1}\rangle \\& = \sum_{k=0}^{m-1}\bigg(  \langle S_{0,k}(X), j_1\ldots j_{n}\rangle\\&+\sum_{p= 2}^{n+1}(-1)^{p} \sum_{J\in C(p-1)} \frac{1}{J!} \langle S_{0,m}(X), j_1\ldots j_{n+1-p}\rangle\langle \mathcal{S}_{k,m}(x),[j_{n}\ldots j_{n+2-p}]_J \rangle\bigg)\delta x_k^{j_{n+1}}\\&+\left(\sum_{p= 2}^{n+1}(-1)^{p+1} \sum_{J\in C(p)} \frac{1}{J!}  \langle S_{0,m}(X), j_1\ldots j_{n+1-p}\rangle\langle \mathcal{S}_{0,m}(x),[j_{n+1}\ldots j_{n+2-p}]_J \rangle\right)\\& = (I) + (II)\,.
\end{split}\]
We treat the terms $(I)$ and $(II)$ separately. Using  the induction hypothesis, we write the term $(II)$ as follows 
\[
\begin{split}(II)&=-\sum_{p= 0}^{n+1} \sum_{J\in C(p)}\sum_{I\in C(n+1-p)} \frac{1}{I!J!}  \langle \mathcal{S}_{0,m}(x), [j_1\ldots j_{n+1-p}]_I\rangle\langle \mathcal{S}_{0,m}(x),[\mathcal{A}(j_{n+2-p}\ldots j_{n+1})]_J \rangle\\&- \sum_{I\in C(n)}\frac{1}{I!} \langle \mathcal{S}_{0,m}(x),[j_{1}\ldots j_{n}]_I \rangle \langle \mathcal{S}_{0,m}(x), j_{n+1} \rangle +  \sum_{I\in C(n+1)} \frac{1}{I!} \langle \mathcal{S}_{0,m}(x),[j_{1}\ldots j_{n+1}]_I \rangle\,, 
\end{split}
\]
where $\mathcal{A}\colon T(\mathfrak{A}_d)\to T(\mathfrak{A}_d)$ is the linear map defined on $T(\mathfrak{A}_d)$ the vector space of finite linear combinations of words in $\mathfrak{A}_d$  and 
$\mathcal{A}(a_1\ldots a_n)= (-1)^{n}a_n\ldots a_1$ called the \emph{shuffle antipodal map}. To understand the first contribution  we first recall three fundamental algebraic results: firstly the map $ S_{0,m}(X)\colon T(\mathfrak{A}_d)\to \mathbb{R}$ is a character with respect to the shuffle product $\qshuffle$ \cite[Theorem 3.4]{Tapia20}, then the application $\Phi_H\colon T(\mathfrak{A}_d) \to T(\mathfrak{A}_d)$ given on any word $w$
\[\Phi_H(w)= \sum_{I\in C(|I|)} \frac{1}{I!}[w]_I\] 
is a Hopf algebra morphism between $(T(\mathfrak{A}_d), \shuffle, \Delta)$ and $(T(\mathfrak{A}_d), \qshuffle, \Delta)$ with $\Delta$ the deconcatenation coproduct \cite[Theorem 3.3]{hoffman2000}. Finally the map $\mathcal{A}$
satisfies for any $n\geq 1$ the antipode relation
\[\shuffle (\text{id}\otimes \mathcal{A})\Delta (a_1\ldots a_n)=0\,.\]
 Using these identities we derive
 \[
\begin{split}&-\sum_{p= 0}^{n+1} \sum_{J\in C(p)}\sum_{I\in C(n+1-p)} \frac{1}{I!J!}  \langle \mathcal{S}_{0,m}(x), [j_1\ldots j_{n+1-p}]_I\rangle\langle \mathcal{S}_{0,m}(x),[\mathcal{A}(j_{n-p+2}\ldots j_{n+1})]_J \rangle\\&=-\langle \mathcal{S}_{0,m}(x),\sum_{p= 0}^{n+1} \Phi_H(j_1\ldots j_{n+1-p})\qshuffle  \Phi_H(\mathcal{A}(j_{n-p+2}\ldots j_{n+1}))\rangle\\& =-\langle \mathcal{S}_{0,m}(x), \qshuffle(\Phi_H\otimes \Phi_H)(\text{id}\otimes \mathcal{A})(j_{1}\ldots j_{n+1}))\rangle\\&= -\langle \mathcal{S}_{0,m}(x), \Phi_H\shuffle(\text{id}\otimes \mathcal{A})(j_{1}\ldots j_{n+1}))\rangle= -\langle \mathcal{S}_{0,m}(x), \Phi_H (0)\rangle=0\,,
\end{split}
\]
from which we conclude
\begin{equation}\label{eq:first_technical_lemma}
\begin{split}(I)=- \sum_{I\in C(n)}\frac{1}{I!} \langle \mathcal{S}_{0,m}(x),[j_{1}\ldots j_{n}]_I \rangle \langle \mathcal{S}_{0,m}(x), j_{n+1} \rangle +  \sum_{I\in C(n+1)} \frac{1}{I!} \langle \mathcal{S}_{0,m}(x),[j_{1}\ldots j_{n+1}]_I \rangle\,.
\end{split}
\end{equation}
To deal with the term $(I)$, we apply again the induction assumption, the definitions of the previous maps and a change of variable in the sum over $p$ to obtain 
\[\begin{split}(I)&=\sum_{k=0}^{m-1}\bigg(\sum_{p= 1}^{n}(-1)^{p+1} \langle \mathcal{S}_{0,m}(x), \Phi_H(j_1\ldots j_{n-p})\rangle\langle \mathcal{S}_{k,m}(x),\Phi_H(j_{n}\ldots j_{n+1-p})\rangle\\&+ \langle \mathcal{S}_{0,m}(x), \Phi_H(j_1\ldots j_n)\rangle \bigg)\delta x_k^{j_{n+1}}\\&=\sum_{k=0}^{m-1}\bigg(-\sum_{p= 0}^{n} \langle \mathcal{S}_{0,m}(x), \Phi_H(j_1\ldots j_{n-p})\rangle\langle \mathcal{S}_{k,m}(x),\Phi_H(\mathcal{A}(j_{n+1-p}\ldots j_{n}))\rangle\\&+ \langle \mathcal{S}_{0,m}(x), \Phi_H(j_1\ldots j_n)\rangle \bigg)\delta x_k^{j_{n+1}}+\sum_{k=0}^{m-1}\langle \mathcal{S}_{0,k}(x), \Phi_H(j_1\ldots j_{n})\rangle\delta x_k^{j_{n+1}}\,.
\end{split}\]
Let us further introduce the quasi-shuffle antipodal map $\widehat{\mathcal{A}}\colon \colon T(\mathfrak{A}_d)\to T(\mathfrak{A}_d)$ defined on a  word $(a_1\ldots a_n)$ by
\[\widehat{\mathcal{A}}(a_1\ldots a_n)= (-1)^n\sum_{I\in C(n)} [a_n\ldots a_1]_I\,.\]
From standard properties on Hopf algebras we derive
\[\Phi_H \circ \mathcal{A}=\widehat{\mathcal{A}} \circ \Phi_H\,, \quad \langle \mathcal{S}_{k,m}(x), \widehat{\mathcal{A}} (a_1\ldots a_n)\rangle= \langle (\mathcal{S}_{k,m}(x))^{-1}, a_1\ldots a_n\rangle\,\]    
where the inverse is taken with respect to the tensor product $\otimes$ operation. Combining these   properties with Chen's identity for iterated sum signatures, see  \cite[Theorem 3.4]{Tapia20}, for any integer $0\leq k\leq m-1$ we obtain
\[\begin{split}&\sum_{p= 0}^{n} \langle \mathcal{S}_{0,m}(x), \Phi_H(j_1\ldots j_{n-p})\rangle\langle \mathcal{S}_{k,m}(x),\Phi_H(\mathcal{A}(j_{n+1-p}\ldots j_{n}))\rangle\\&= \sum_{p= 0}^{n} \langle \mathcal{S}_{0,m}(x), \Phi_H(j_1\ldots j_{n-p})\rangle\langle \mathcal{S}_{k,m}(x),\widehat{\mathcal{A}} (\Phi_H(j_{n+1-p}\ldots j_{n}))\rangle\\& =\sum_{p= 0}^{n} \langle \mathcal{S}_{0,m}(x), \Phi_H(j_1\ldots j_{n-p})\rangle\langle (\mathcal{S}_{k,m}(x))^{-1},(\Phi_H(j_{n+1-p}\ldots j_{n}))\rangle\\&=\langle \mathcal{S}_{0,m}(x)\otimes(\mathcal{S}_{k,m}(x))^{-1} ,\Phi_H(j_{1}\ldots j_{n})\rangle=\langle \mathcal{S}_{0,k}(x) ,\Phi_H(j_{1}\ldots j_{n})\rangle\,.
\end{split}\]
Plugging this expression into the definition of $(II)$  we finally obtain 
\begin{equation}\label{eq:second_technical_lemma}
\begin{split}(II)=\sum_{k=0}^{m-1}\langle \mathcal{S}_{0,m}(x), \Phi_H(j_1\ldots j_{n})\rangle\delta x_k^{j_{n+1}}= \sum_{I\in C(n)}\frac{1}{I!} \langle \mathcal{S}_{0,m}(x),[j_{1}\ldots j_{n}]_I \rangle \langle \mathcal{S}_{0,m}(x), j_{n+1} \rangle \,.
\end{split}
\end{equation}
Adding the two expressions \eqref{eq:first_technical_lemma} and \eqref{eq:second_technical_lemma} yields the desired quantity.	\end{proof}

	\bibliographystyle{alpha}
	\bibliography{bibliography}
	
\end{document}